\documentclass[12pt]{article}
\usepackage[margin=0.5in]{geometry}
\usepackage[utf8]{inputenc}
\usepackage[english]{babel}
\usepackage{amsthm}
\usepackage{graphicx}
\usepackage{url}
\usepackage{hyperref}
\usepackage{amsmath}
\usepackage{listings}
\usepackage{amsfonts}
\usepackage{xcolor}
\usepackage[numbib]{tocbibind}
\newtheorem{theorem}{Theorem}

\begin{document}

\title{The inverse Riemann zeta function}

\author{Artur Kawalec}

\date{}
\maketitle

\begin{abstract}
In this article, we develop a formula for an inverse Riemann zeta function such that for $w=\zeta(s)$ we have $s=\zeta^{-1}(w)$ for real and complex domains $s$ and $w$. The presented work is based on extending the analytical recurrence formulas for trivial and non-trivial zeros as to solve an equation $\zeta(s)-w=0$ for a given $w$-domain using logarithmic differentiation and zeta recursive root extraction methods. We further explore formulas for trivial and non-trivial zeros of the Riemann zeta function in greater detail, and next, we also explore an expansion of the inverse zeta function by an attractor of its branch singularities, and develop some identities that emerge from them.  In the last part, we extend the presented results as a general method for finding zeros and inverses of many other functions, such as the gamma function, the Bessel function of the first kind, or finite/infinite degree polynomials and rational functions, etc. We further compute all the presented formulas numerically to high precision and show that these formulas do indeed converge to the inverse of the Riemann zeta function and the related results. We also develop a fast algorithm to compute $\zeta^{-1}(w)$ for complex $w$.
\end{abstract}

\tableofcontents
\newpage
\section{Introduction}
The Riemann zeta function is classically defined by an infinite series
\begin{equation}\label{eq:20}
\zeta(s) = \sum_{n=1}^{\infty}\frac{1}{n^s},
\end{equation}
which is absolutely convergent $\Re(s)>1$, where $s=\sigma+it$ is a complex variable. The values for the first few special cases are:

\begin{equation}\label{eq:9}
\begin{aligned}
\zeta(1) &\sim\sum_{n=1}^{k}\frac{1}{n}\sim\gamma+\log(k) \quad \text{as}\quad k\to \infty,\\
\zeta(2) &=\frac{\pi^2}{6}, \\
\zeta(3) &=1.20205690315959\dots, \\
\zeta(4) &=\frac{\pi^4}{90}, \\
\zeta(5) &=1.03692775514337\dots,
\end{aligned}
\end{equation}
and so on. For $s=1$, the series diverges asymptotically as $\gamma+\log(k)$, where $\gamma=0.5772156649\dots$ is the Euler-Mascheroni constant. The special values for even positive integer argument are generated by the Euler's formula
\begin{equation}\label{eq:9}
\zeta(2k) = \frac{\mid B_{2k}\mid}{2(2k)!}(2\pi)^{2k},
\end{equation}
for which the value is expressed as a rational multiple of $\pi^{2k}$ where the constants $B_{2k}$ are Bernoulli numbers defined such that $B_0=1$, $B_1=-\frac{1}{2}$, $B_2=\frac{1}{6}$, and so on. For odd positive integer argument, the values of $\zeta(s)$ converge to unique constants, which are not known to be expressed as a rational multiple of $\pi^{2k+1}$ as occurs in the even positive integer case. For $s=3$, the value is commonly known as  Ap\'ery's constant, who proved its irrationality [2]. The key connection to prime numbers is by means of Euler's infinite product formula

\begin{equation}\label{eq:20}
\zeta(s) = \prod_{n=1}^{\infty}\left(1-\frac{1}{p_n^s}\right)^{-1}
\end{equation}
where $p_1=2$, $p_2=3$, $p_3=5$ and so on, denote the prime number sequence. These prime numbers can be recursively extracted from the Euler product by the Golomb's formula [9]. Hence, if we define a partial Euler product up to the $n$th order as
\begin{equation}\label{eq:20}
Q_{n}(s)=\prod_{k=1}^{n}\left(1-\frac{1}{p_k^s}\right)^{-1}
\end{equation}
for $n>1$ and $Q_0(s)=1$, then we obtain a recurrence relation for the $p_{n+1}$ prime
\begin{equation}\label{eq:20}
p_{n+1}=\lim_{s\to \infty}\left(1-\frac{Q_n(s)}{\zeta(s)}\right)^{-1/s}.
\end{equation}
This leads to representation of primes by the following limit identities

\begin{equation}\label{eq:20}
\begin{aligned}
p_{1} &=\lim_{s\to \infty}\left[1-\frac{1}{\zeta(s)}\right]^{-1/s}, \\
\\
p_{2} &=\lim_{s\to \infty}\left[1-\frac{(1-\frac{1}{2^s})^{-1}}{\zeta(s)}\right]^{-1/s}, \\
\\
p_{3} &=\lim_{s\to \infty}\left[1-\frac{(1-\frac{1}{2^s})^{-1}(1-\frac{1}{3^s})^{-1}}{\zeta(s)}\right]^{-1/s}, \\
\end{aligned}
\end{equation}
and so on, whereby all the previous primes are used to excite the Riemann zeta function in a such as way as to extract the next prime. A detailed proof and numerical computation is shown in [11][12]. We will find that this recursive structure (when taken in the limit as $s\to \infty$) is a basis for the rest of this article, and will lead to formulas for trivial and non-trivial zeros, and the inverse Riemann zeta function.

Furthermore, the Riemann zeta series (1) induces a general Weierstrass factorization of the form
\begin{equation}\label{eq:20}
\zeta(s)=\frac{e^{(\log(2\pi)-1)s}}{2(s-1)}\prod_{\rho_{t,n}}^{}\left(1-\frac{s}{\rho_{t,n}}\right)e^{\frac{s}{\rho_{t,n}}}\prod_{\rho_{nt}}^{}\left(1-\frac{s}{\rho_{nt}}\right)e^{\frac{s}{\rho_{nt}}}
\end{equation}
where it analytically extends the zeta function to the whole complex plane and reveals its full structure of the poles and zeros [1, p.807]. Only a simple pole exists at $s=1$, hence $\zeta(s)$ is convergent everywhere else in the complex plane, i.e., the set $\mathbb{C}\backslash 1$.  Moreover, there are two kinds of zeros classified as the trivial zeros $\rho_t$ and non-trivial zeros $\rho_{nt}$.  The first infinite product term of (8) encodes the factorization due to trivial zeros

\begin{equation}\label{eq:20}
\rho_{t,n}=-2n \quad
\end{equation}
for $n\geq 1$ which occur at negative even integers $-2,-4,-6 \ldots$, where $n$ is index variable for the $n$th zero. The second infinite product term of (8) encodes the factorization due to non-trivial zeros, which are complex numbers of the form

\begin{equation}\label{eq:20}
\rho_{nt,n}=\sigma_n+it_n
\end{equation}
and, as before, $n$ is the index variable for the $n$th zero (this convention is straightforward if the zeros are on the critical line). But in general, the real components of non-trivial zeros are known to be constrained to lie in a critical strip in a region between $0<\sigma_n<1$. It is also known that there is an infinity of zeros located on the critical line at $\sigma=\frac{1}{2}$, but it is not known whether there are any zeros off of the critical line, a problem of the Riemann hypothesis (RH), which proposes that all zeros should lie on the critical line. The first few zeros on the critical line at $\sigma_n=\frac{1}{2}$ have imaginary components $t_1 = 14.13472514...$, $t_2 = 21.02203964...$, $t_3 = 25.01085758...$, and so on,  which were computed by an analytical recurrence formula as

\begin{equation}\label{eq:20}
t_{n+1} = \lim_{m\to\infty}\left[\frac{(-1)^{m}}{2}\left(2^{2m}-\frac{1}{(2m-1)!}\log (|\zeta|)^{(2m)}\big(\frac{1}{2}\big)-\frac{1}{2^{2m}}\zeta(2m,\frac{5}{4})\right)-\sum_{k=1}^{n}\frac{1}{t_{k}^{2m}}\right]^{-\frac{1}{2m}}
\end{equation}
as we have shown in [12][13] assuming (RH). Also, $\zeta(s,a)$ is the Hurwitz zeta function
\begin{equation}\label{eq:20}
\zeta(s,a) = \sum_{n=0}^{\infty} \frac{1}{(n+a)^s},
\end{equation}
which is a shifted version of (1) by an arbitrary parameter $a\neq -1,-2,\ldots$. Now, substituting the Weierstrass infinite product (8) into the Golomb's formula (6) can be used to generate primes directly from non-trivial zeros. In later sections, we will show that non-trivial zeros can also be generated directly from primes.

Furthermore, the Riemann zeta function can have many points $s_n$ such that $w=\zeta(s_n)$ can map to the same $w$ value. In Figure 1 we plot $\zeta(s)$ for real $s$ and $w$, and note that for $s>1$ the function is monotonically decreasing from $+\infty$ and tends $O(1)$ as $s\to \infty$, and for the domain $-2.7172628292\ldots<s<1$ it's monotonically decreasing from $0.0091598901\ldots$ to $-\infty$, as also shown in Figure 2. And for $s<-2.7172628292\ldots$, it becomes oscillatory where there are many $s_n$ solutions. For example, the first two $s$ values
\begin{equation}\label{eq:20}
\begin{aligned}
\zeta(-2.47270347\ldots) &= \zeta(-3)=\frac{1}{120}
\end{aligned}
\end{equation}
map to the same $w$ value as shown in Figure 2. It is usually customary to report that $\zeta(-3)=\frac{1}{120}$, but in fact is the second solution $s_2$, the first solution, or the principal solution $s_1$, is the value $-2.47270347\ldots$.

\begin{figure}[h]
  \centering
  \includegraphics[width=170mm]{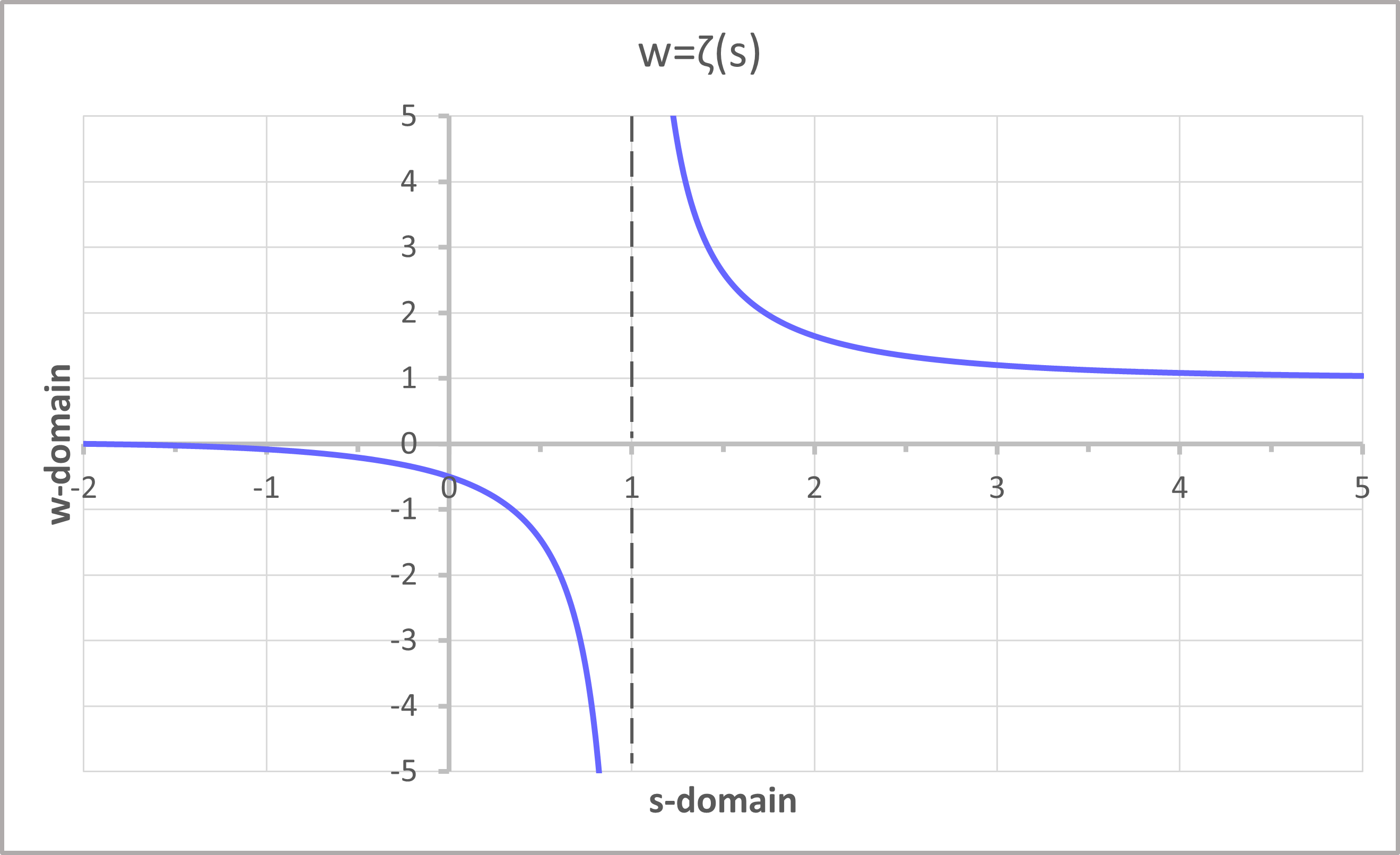}\\
  \caption{A plot of $w=\zeta(s)$ for $s\in (-2,5)$. }\label{1}
\end{figure}

\begin{figure}[h]
  \centering
  \includegraphics[width=170mm]{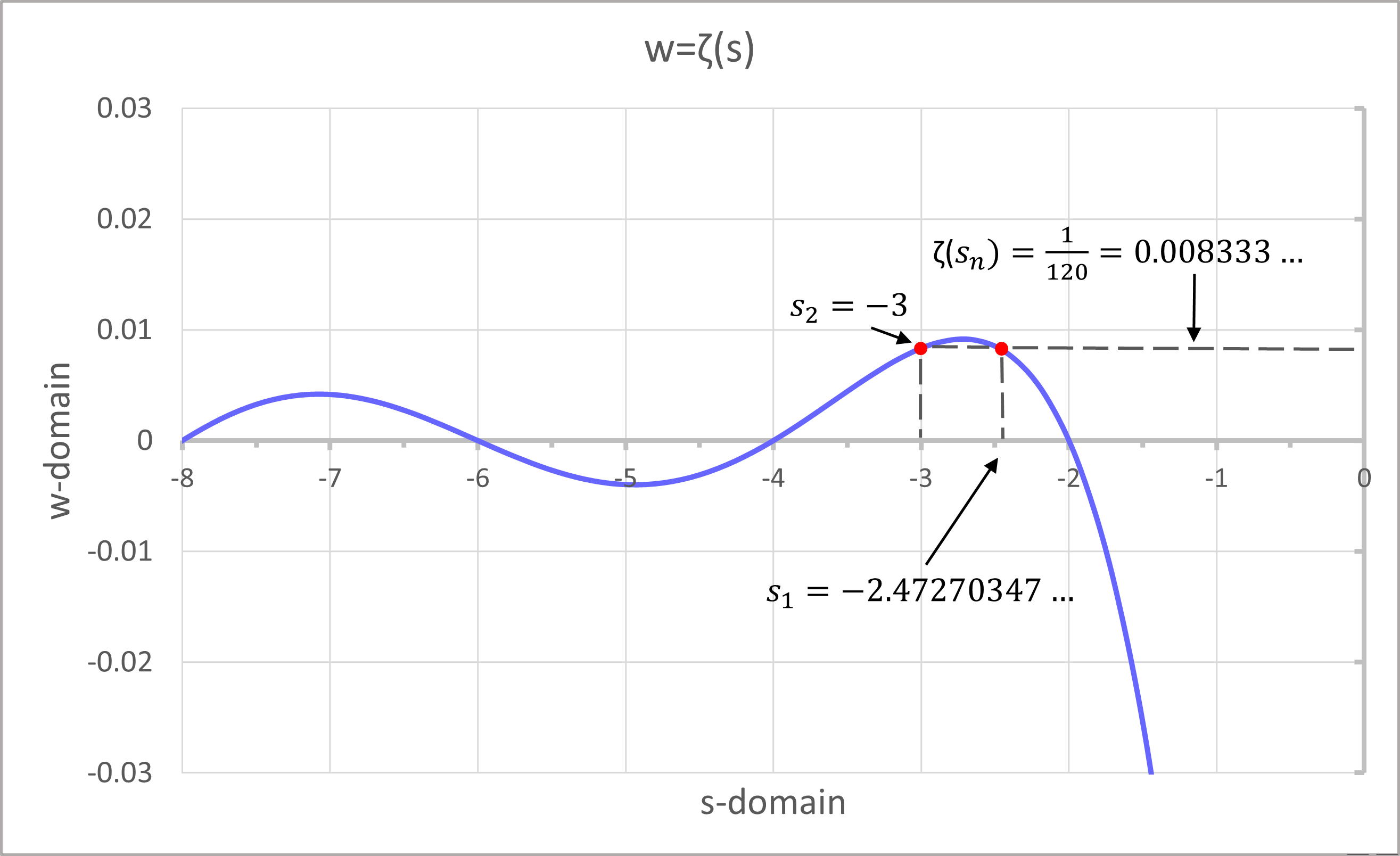}\\
  \caption{A zoomed-in plot of $w=\zeta(s)$ for $s\in (-8,0)$ showing oscillatory behavior. }\label{1}
\end{figure}

In this article, we seek to develop a formula for an inverse Riemann zeta function $s=\zeta^{-1}(w)$. In general, the existence of an inverse is established by an inverse function theorem as shown in [14, p.135], where for any holomorphic function $f(z)$, an inverse exists when $f'(z)\neq 0$ in the $z$-domain. Hence, for

\begin{equation}\label{eq:20}
w=\zeta(s)
\end{equation}
there is an inverse function
\begin{equation}\label{eq:20}
s=\zeta^{-1}(w),
\end{equation}
which implies that

\begin{equation}\label{eq:20}
\zeta^{-1}(\zeta(s)) = s
\end{equation}
and
\begin{equation}\label{eq:20}
\zeta(\zeta^{-1}(w))=w
\end{equation}
for some real and complex domains $w$ and $s$, except at those points where $\zeta'(s)=0$ in the $s$-domain. The inverse zeta is a multi-valued function with infinite number of branches, and $s_n$ are the multi-valued solutions for which $w=\zeta(s_n)$.  The presented method can also recursively access these multi-valued solutions of $s_n=\zeta^{-1}(w)$, but as we will find, the computational precision becomes very high, and so we will primarily focus on the principal solution $s_1$, which as we will find is very easy to compute and will cover almost the entire complex plane. Additionally, in [8] discusses additional theoretical basis behind solutions to (14), also known as $a$-points, which we refer to as $s_n$.

We develop a recursive formula for an inverse Riemann zeta function as:

\begin{equation}\label{eq:20}
s_{n+1}=\zeta^{-1}(w)= \lim_{m\to\infty}\pm \Bigg[-\frac{1}{(2m-1)!}\frac{d^{(2m)}}{ds^{(2m)}}\log\Big[(\zeta(s)-w)(s-1)\Big]\Bigr\rvert_{s\to 0}-\sum_{k=1}^{n}\frac{1}{s_k^{2m}}\Bigg]^{-\frac{1}{2m}}
\end{equation}
where $s_n$ is the value for which $w=\zeta(s_n)$ for some branch $\lambda$ of the $m$th root; it will recursively generate solutions $s_n$ for a given $w$-domain. The formula for the principal branch is

\begin{equation}\label{eq:20}
s_{1}=\zeta^{-1}(w)=\lim_{m\to\infty}\pm \left[-\frac{1}{(m-1)!}\frac{d^{m}}{ds^{m}}\log\left[(\zeta(s)-w)(s-1)\right]\Bigr\rvert_{s\to 0}\right]^{-\frac{1}{m}},
\end{equation}
and is valid for all complex $w$-domain except in a small strip region $\Re(w)\in (j_1,1)\cup \{\Im(w)=0\}$ that we determined experimentally, and $j_1=0.00915989\ldots$ is a constant. As we will show in more details in later sections, this formula can easily invert the Basel problem

\begin{equation}\label{eq:20}
\zeta^{-1}\left(\frac{\pi^2}{6}\right)=2
\end{equation}
or the Ap\'ery's  constant

\begin{equation}\label{eq:20}
\zeta^{-1}(1.20205690315959428\ldots)=3,
\end{equation}
and essentially the entire complex domain $w\in \mathbb{C}$ with an exception of a strip region $\Re(w)\in (j_1,1)\cup \{\Im(w)=0\}$ where there lie (possibly) an infinite number of branch non-isolated singularities.

The way in which we will arrive at the presented formula is by connecting two simple Theorems. The first Theorem 1, as presented in Section 2, is the $m$th log-derivative formula for obtaining a generalized zeta series over the zeros of a function. Such a method appears in the literature from time to time and can be traced back to Euler, who, according to a reference in [21, p.500], used it to devise a means of computing several zeros of the Bessel function of the first kind by solving a system of equations generated by the log-derivative formula. In the works of Voros, Lehmer, and Matsuoka, it has been used to find a closed-form formula for the secondary zeta functions [15][16][19][20]. Then, the second Theorem 2, as given in Section 3, develops a method for finding a zeta recurrence formula for the $n$th+1 term of a generalized zeta series, i.e., all terms of a generalized zeta series must be known in order to generate the $n$th+1 term, as we have shown in our previous work [12][13]. In Section 4, we connect these two theorems and find formulas for trivial and non-trivial zeros of the Riemann zeta function and explore their properties in greater detail. We then develop a formula for an inverse zeta and also explore some its identities. We empirically observe that there are infinitely many branch singularities of the inverse zeta that are spread out along a narrow strip $(j_1,1)$ forming branch singularity attractor. And in the last part, we briefly extend the presented results to find zeros (and inverses) of many common functions, such as the gamma function, the Bessel function of the first kind, the trigonometric functions, Lambert-W function, and any entire function or finite/infinite degree polynomial or rational function (provided that the function fits the constraints of this method).

Throughout this article, we numerically compute the presented formulas to high precision in PARI/GP software package [17], as it is an excellent platform for performing arbitrary precision computations, and show that these formulas do indeed converge to the inverse of the Riemann zeta function to high precision. We note that when running the script in PARI, the precision has to be set very high (we generally set precision to 1000 decimal places). Also, the Wolfram Mathematica software package [22] was instrumental in developing this article.

\section{Logarithmic differentiation}
In this section, we outline the zeta $m$th log-derivative method. We recall the argument principle, that for an analytic function $f(z)$ we have

\begin{equation}\label{eq:20}
\frac{1}{2\pi i}\oint_{\Omega} \frac{f^{\prime}(z)}{f(z)} dz = N_z-N_p
\end{equation}
which equals the number of zeros $N_z$ minus the number of poles $N_p$ (counting multiplicity) which are enclosed in a simple contour $\Omega$. But now, instead working with the number of zeros or poles, the aim of the $m$th log-derivative method is to find a generalized zeta series over the zeros and poles of the function $f(z)$ in question. Hence if $Z=\{z_1,z_2,z_3,\ldots,z_{N_z}\}$ is a set of all zeros of $f(z)$ and $P=\{p_1,p_2,p_3,\ldots, p_{N_p} \}$ is a set of all poles of $f(z)$ in the whole complex plane, then the generalized zeta series are

\begin{equation}\label{eq:20}
Z(s) = \sum_{n=1}^{N_z}\frac{1}{z_n^s}
\end{equation}
and
\begin{equation}\label{eq:20}
P(s) = \sum_{n=1}^{N_p}\frac{1}{p_n^s}.
\end{equation}
The number of zeros or poles may be finite or infinite, but in the latter case, the convergence of a generalized zeta series to the right-half side of the line $\Re{(s)}>\mu$ may depend on the distribution of its terms. If the number of zeros and poles is finite, then one could count them by evaluating $Z(0)-P(0)=N_z-N_p$, which reduces to the argument principle (22). The argument principle may be modified further by introducing another function $h(z)$ as such

\begin{equation}\label{eq:20}
\frac{1}{2\pi i}\oint_{\Omega}  \frac{f^{\prime}(z)}{f(z)}h(z) dz = \sum_{n=1}^{N_z}h(z_i)-\sum_{n=1}^{N_p}h(p_n).
\end{equation}
The proof is a slight modification of a standard proof of (22) using the Residue theorem, and if we let $h(z)=z^{-s}$, then we have
\begin{equation}\label{eq:20}
\frac{1}{2\pi i}\oint_{\Omega}  \frac{f^{\prime}(z)}{f(z)}\frac{1}{z^s} dz = Z(s)-P(s)
\end{equation}
but the contour has to be deformed as to not encircle the origin. However, in the following sections, we will not pursue (26) as the integral makes the study and computation more difficult, as well as being dependent on the contour. There is a simpler variation of (26) which is better suited for our study. The main formula of the zeta $m$th log-derivative method is:

\begin{theorem}
If $f(z)$ is meromorphic function with a set of zeros $z_i$ and poles $p_i$ in $\mathbb{C}\backslash 0$, then we have
\begin{equation}\label{eq:20}
-\frac{1}{(m-1)!}\frac{d^{m}}{dz^{m}}\log\left[\frac{f(z)}{e^{g(z)}}\right]\Bigr\rvert_{z\to 0}=Z(m)-P(m),
\end{equation}
which is valid for a positive integer variable $m>a$, where $a$ is a constant and $e^{g(z)}$ is the associated zero-free scaling function of the Weierstrass factorization products.
\end{theorem}
\noindent This formula generates $Z(m)-P(m)$ over all zeros and poles in the whole complex plane, as opposed to being enclosed in some contour as in (26). We now outline a proof of Theorem 1.

\begin{proof}
We express a meromorphic function $f(z)$ as having simple zeros and poles by the ratio of two entire functions that are admitting a Weierstrass factorization of the form

\begin{equation}\label{eq:20}
f(z)=e^{g(z)}\prod_{n=1}^{\infty}\left(1-\frac{z}{z_n}\right)\prod_{n=1}^{\infty}\left(1-\frac{z}{p_n}\right)^{-1}\phi_a(z; z_n,p_n),
\end{equation}
with the zero or pole at origin removed, and $e^{g(z)}$ is a scaling component not having any zeros or poles, and $\exp[\phi_a(z;z_n,p_n)]=\frac{z}{z_n}+\frac{1}{2}\frac{z^2}{z^2_n}+\ldots+\frac{1}{u}\frac{z^{u}}{z^u_{n}}-\frac{z}{p_n}-\frac{1}{2}\frac{z^2}{p^2_n}-\ldots-\frac{1}{v}\frac{z^{v}}{p^v_{n}}$ is the exponential part of the canonical factors up to maximum $a=\max(u,v)$ order, then one has

\begin{equation}\label{eq:20}
\log\left(\frac{f(z)}{e^{g(z)}}\right)=\phi_a(z;z_n,p_n)+\sum_{n=1}^{\infty}\log\left(1-\frac{z}{z_n}\right)-\sum_{n=1}^{\infty}\log\left(1-\frac{z}{p_n}\right).
\end{equation}
Now, using the Taylor series expansion for the logarithm as
\begin{equation}\label{eq:20}
\log(1-z)=-\sum_{k=1}^{\infty}\frac{z^k}{k}=-z-\frac{z^2}{2}-\frac{z^3}{3}-\ldots
\end{equation}
for $|z|<1$, we obtain

\begin{equation}\label{eq:20}
\log\left(\frac{f(z)}{ e^{g(z)}}\right)=\phi_a(z;z_n,p_n)-\sum_{n=1}^{\infty}\sum_{k=1}^{\infty}\frac{1}{k}\left(\frac{z}{z_n}\right)^k+\sum_{n=1}^{\infty}\sum_{k=1}^{\infty}\frac{1}{k}\left(\frac{z}{p_n}\right)^k.
\end{equation}
Now interchanging the order of summation yields

\begin{equation}\label{eq:20}
\log\left(\frac{f(z)}{ e^{g(z)}}\right)=\phi_a(z;z_n,p_n)-\sum_{k=1}^{\infty}\frac{z^k}{k}\left[\sum_{n=1}^{\infty}\frac{1}{z^k_n}\right]+\sum_{k=1}^{\infty}\frac{z^k}{k}\left[\sum_{n=1}^{\infty}\frac{1}{p^k_n}\right],
\end{equation}
and hence, by recognizing the inner sum as a generalized zeta series yields

\begin{equation}\label{eq:20}
\log\left(\frac{f(z)}{e^{g(z)}}\right)=\phi_a(z;z_n,p_n)-\sum_{k=1}^{\infty}Z(k)\frac{z^k}{k}+\sum_{k=1}^{\infty}P(k)\frac{z^k}{k}.
\end{equation}
From this form, we can now extract $Z(m)-P(m)$ by the $m^{th}$ order differentiation as

\begin{equation}\label{eq:20}
-\frac{1}{(m-1)!}\frac{d^{m}}{dz^{m}}\log\left[\frac{f(z)}{e^{g(z)}}\right]\Bigr\rvert_{z\to 0}=Z(m)-P(m)
\end{equation}
evaluated as $z\to 0$ in the limit, as the derivatives of $(\phi_a(z;z_n,p_n))^{(m)}=0$ vanish for $m>a$.

\end{proof}

We also remark if the number of zeros or poles is infinite, then the Weierstrass factorization product naturally imposes that $Z(m)$ and $P(m)$ are absolutely convergent, which is very critical for next section.

Additionally, we can obtain an integral representation using the Cauchy integral formula applied to coefficients of Taylor expansion (34) as
\begin{equation}\label{eq:20}
\lim_{z_0\to 0}\Bigg\{-\frac{m}{2\pi i}\oint_{\Omega_0} \frac{1}{(z-z_0)^{m+1}}\log\left(\frac{f(z)}{e^{g(z)}}\right)dz\Bigg\}=Z(m)-P(m),
\end{equation}
where $\Omega_0$ is a simple contour encircling the origin, but unlike (26), is not enclosing zeros or poles. In order to apply (34) and (35) successfully, one has to judiciously choose $g(z)$ as to remove it from $f(z)$ so that the $m$th log-differentiation will not produce unwanted artifacts due to $g(z)$, as well as the zero or pole at origin.

\section{The zeta recurrence formula}
We now outline a method to extract the terms of a generalized zeta series by means of a recurrence formula satisfied by the terms of such series. Hence, if the terms of a generalized zeta series are to be zeros of a function, then such a method effectively gives a recurrence formula satisfied by the zeros, which in turn can be used to compute the zeros. And similarly, the same holds if the terms of a generalized zeta series are poles of a function. In fact, any quantities represented by the terms of a generalized zeta series can be recursively found. In [12], we developed a formula for the $n$th+1 prime based on the prime zeta function. But for the purpose of this paper, let us consider the generalized zeta series over zeros $z_n$ of a function

\begin{equation}\label{eq:20}
Z(s) = \sum_{n=1}^{N_z}\frac{1}{z_n^s} = \frac{1}{z_1^s}+\frac{1}{z_2^s}+\frac{1}{z_3^s}+\ldots
\end{equation}
and let us also assume that the zeros are positive, real, and ordered from smallest to largest such that $0<z_1<z_2<z_3<\ldots <z_n$, then the asymptotic relationship holds
\begin{equation}\label{eq:20}
\frac{1}{z_n^s} \gg \frac{1}{z_{n+1}^s}
\end{equation}
as $s\to \infty$. To illustrate how fast the terms decay, let us take $z_1=2$ and $z_2$=3. Then for $s=10$ we compute

\begin{equation}\label{eq:20}
\begin{aligned}
\frac{1}{2^s} &= 9.7656\ldots\times 10^{-4},\\
\frac{1}{3^s} &= 1.6935\ldots\times 10^{-5},
\end{aligned}
\end{equation}
where we roughly see an order of magnitude difference. But for $s=100$ we compute
\begin{equation}\label{eq:20}
\begin{aligned}
\frac{1}{2^s} &= 7.8886\ldots\times 10^{-31},\\
\frac{1}{3^s} &= 1.9403\ldots\times 10^{-48},
\end{aligned}
\end{equation}
where we see a difference by $17$ orders of magnitude. Hence, as $s\to \infty$, then

\begin{equation}\label{eq:20}
O\left(z_{n}^{-s}\right) \gg O\left(z_{n+1}^{-s}\right)
\end{equation}
as the $z^{-s}_{n+1}$ term completely vanishes in relation to $z^{-s}_{n}$, and so the $z^{-s}_{n}$ term dominates the limit. As a result, we write

\begin{equation}\label{eq:20}
z_{1}^{-s} \gg z_{2}^{-s} \gg z_{3}^{-s} \gg \ldots \gg z_{n}^{-s}.
\end{equation}
From this we have

\begin{equation}\label{eq:20}
Z(s) \sim O\left(z_{1}^{-s}\right)
\end{equation}
as $s\to \infty$ where the lowest order term dominates, and we refer to it as the principal term, or in the case where the zeta series are considered to be zeros of a function, the principal zero. This rapid decay of higher order zeta terms (41) opens a possibility for a recursive root extraction as shown by Theorem 2 next.

\begin{theorem}
If $\{z_n\}$ is a set of positive real numbers ordered such that $0<z_1<z_2<z_3<\ldots <z_n$, and so on, then the recurrence relation for the $n$th+1 term is
\begin{equation}\label{eq:20}
z_{n+1}=\lim_{s\to\infty}\left(Z(s)-\sum_{k=1}^{n}\frac{1}{z_k^{s}}\right)^{-1/s}.
\end{equation}
\end{theorem}

\begin{proof}
First we begin solving for $z_{1}$ in (36) to obtain
\begin{equation}\label{eq:20}
\frac{1}{z_1^s}=Z(s)-\frac{1}{z_2^s}-\frac{1}{z_3^s}-\ldots
\end{equation}
and then we get
\begin{equation}\label{eq:20}
z_1=\left(Z(s) -\frac{1}{z_2^s}-\frac{1}{z_3^s}-\ldots\right)^{-1/s}.
\end{equation}
If we then consider the limit
\begin{equation}\label{eq:20}
z_1=\lim_{s\to\infty}\left(Z(s)-\frac{1}{z_2^s}-\frac{1}{z_3^s}-\ldots\right)^{-1/s}
\end{equation}
then because the series is convergent and since $z_1^{-s}\gg z_2^{-s}$, then $O[Z(s)]\sim O(z_1^{-s})$ as $s\to \infty$, and so the higher order zeros decay as $O(z_2^{-s})$ faster than $Z(s)$, and so $Z(s)$ dominates the limit, hence the formula for the principal zero is:
\begin{equation}\label{eq:20}
z_1=\lim_{s\to\infty}\left[Z(s)\right]^{-1/s}.
\end{equation}

\noindent The next zero is found the same way, by solving for $z_2$ in (36) we get

\begin{equation}\label{eq:20}
z_2=\lim_{s\to\infty}\left(Z(s)-\frac{1}{z_1^{s}}-\frac{1}{z_3^{s}}-\ldots\right)^{-1/s}
\end{equation}
and since the higher order zeros decay as $z_3^{-s}$ faster than $Z(s)-z_1^{-s}$, we then have

\begin{equation}\label{eq:20}
z_2=\lim_{s\to\infty}\left(Z(s)-\frac{1}{z_1^{s}}\right)^{-1/s}.
\end{equation}
And continuing on to next zero, by solving for $z_3$ in (36) and by removing the next dominant terms, we obtain

\begin{equation}\label{eq:20}
z_3=\lim_{s\to\infty}\left(Z(s)-\frac{1}{z_1^{s}}-\frac{1}{z_2^{s}}\right)^{-1/s},
\end{equation}
and the process continues for the next zero. Hence in general, the recurrence formula for the $n$th+1 zero is

\begin{equation}\label{eq:20}
z_{n+1}=\lim_{s\to\infty}\left(Z(s)-\sum_{k=1}^{n}\frac{1}{z_k^{s}}\right)^{-1/s},
\end{equation}
thus all zeros up to the $n$th order must be known in order to generate the $n$th+1 zero.
\end{proof}

Let us next give an example of how Theorem 1 and Theorem 2 is applied to find a formula for zeros for any meromorphic function. For simplicity, suppose that we wish to find zeros of the function

\begin{equation}\label{eq:20}
f(s)= \frac{\sin(\pi s)}{\pi s}=0.
\end{equation}
We know in advance that the zeros are just integer multiples: $z_n=\pm n$ (for any non-zero integer $n=1,2,3\ldots$) and zero at origin is already canceled. But now, if we apply the $m$th log-derivative formula to (52), then we get a generalized zeta series of over the zeros as

\begin{equation}\label{eq:20}
Z(m)=-\frac{1}{(m-1)!}\frac{d^{m}}{ds^{m}}\log\left[\frac{\sin(\pi s)}{\pi s}\right]\Bigr\rvert_{s\to 0}
\end{equation}
which is equal to the zeta series over all zeros (including the negative ones) as such

\begin{equation}\label{eq:20}
Z(m) =\ldots+\frac{1}{(-3)^m}+\frac{1}{(-2)^m}+\frac{1}{(-1)^m}+\frac{1}{(1)^m}+\frac{1}{(2)^m}+\frac{1}{(3)^m}+\ldots.
\end{equation}
From this we can deduce that even values become double of a half the side of zeros (which in this example is equivalent to $\zeta(s)$) as
\begin{equation}\label{eq:20}
Z(2m) = 2\zeta(2m)
\end{equation}
and the odd values cancel
\begin{equation}\label{eq:20}
Z(2m+1) = 0.
\end{equation}
In view of this, the Euler's formula (3) is
\begin{equation}\label{eq:9}
\zeta(2k) = \frac{\mid B_{2k}\mid}{2(2k)!}(2\pi)^{2k},
\end{equation}
is an example of a closed-form representation of $\zeta(2k)$ that is not involving zeros directly, i.e., the positive integers. But such a formula may not always be available, so we will just use the $m$th log-derivative formula (53) as the main closed-form representation of $Z(s)$. Also, on a side note, the Bernoulli numbers are expansions coefficients of another function

\begin{equation}\label{eq:9}
\frac{x}{e^x-1}=\sum_{n=0}^{\infty}B_n\frac{x^n}{n!}
\end{equation}
where they can be similarly obtained by $m$th differentiation:

\begin{equation}\label{eq:9}
B_{m}=\lim_{x\to 0}\Big\{\frac{d^{m}}{dx^{m}}\frac{x}{e^x-1}\Big\}.
\end{equation}
Hence in essence, the Euler's formula (57) is just a transformed version of (53), but it just happens that Bernoulli numbers are rational constants. And so, when putting this together, we obtain a full solution to (52) as a recurrence relation

\begin{equation}\label{eq:20}
z_{n+1}=\lim_{m\to\infty}\left[-\frac{1}{2(2m-1)!}\frac{d^{(2m)}}{ds^{(2m)}}\log\left[\frac{\sin(\pi s)}{\pi s}\right]\Bigr\rvert_{s\to 0}-\sum_{k=1}^{n}\frac{1}{z_k^{2m}}\right]^{-\frac{1}{2m}}
\end{equation}
where a $2m$ limit value ensures that it is even, and so, an additional factor of $\frac{1}{2}$ is added due to (54).
The principal zero is:
\begin{equation}\label{eq:20}
z_{1}=\lim_{m\to\infty}\left[-\frac{1}{2(2m-1)!}\frac{d^{(2m)}}{ds^{(2m)}}\log\left[\frac{\sin(\pi s)}{\pi s}\right]\Bigr\rvert_{s\to 0}\right]^{-\frac{1}{2m}}
\end{equation}
and a numerical computation for $m=20$ yields

\begin{equation}\label{eq:20}
z_{1}=0.999999999999\underline{9}7726263\dots,
\end{equation}
which is accurate to 13 digits after the decimal place, and the script to compute it in PARI is presented in Listing 1. The key aspect of the script is the $\textbf{derivnum}$ function for computing the $m$th derivative very accurately which will be very useful for the rest of this article. The next zero is recursively found as
\begin{equation}\label{eq:20}
z_{2}=\lim_{m\to\infty}\left[-\frac{1}{2(2m-1)!}\frac{d^{(2m)}}{ds^{(2m)}}\log\left[\frac{\sin(\pi s)}{\pi s}\right]\Bigr\rvert_{s\to 0}-\frac{1}{1^{2m}}\right]^{-\frac{1}{2m}},
\end{equation}
but we must know the first zero in advance in order to compute the next zero. A numerical computation for $m=20$ yields

\begin{equation}\label{eq:20}
z_{2}=1.9999999\underline{9}547806838689\dots
\end{equation}
which is accurate to 8 digits after the decimal place. Then, the next zero in the sequence is
\begin{equation}\label{eq:20}
z_{3}=\lim_{m\to\infty}\left[-\frac{1}{2(2m-1)!}\frac{d^{(2m)}}{ds^{(2m)}}\log\left[\frac{\sin(\pi s)}{\pi s}\right]\Bigr\rvert_{s\to 0}-\frac{1}{1^{2m}}-\frac{1}{2^{2m}}\right]^{-\frac{1}{2m}}
\end{equation}
but we must know the first two zeros in advance in order to compute the next zero. A numerical computation for $m=20$ yields

\begin{equation}\label{eq:20}
z_{3}=2.99999\underline{9}24565967669286\dots
\end{equation}
which is accurate to 6 digits after the decimal place, and so on. We see that the accuracy becomes lesser for higher zeros, and so the limit variable $m$ has to be increased to get better accuracy. One can then continue this process and extract the $z_{n+1}$ zero.

As a second example, we find a formula for zeros of the Bessel function of the first kind

\begin{equation}\label{eq:20}
J_{\nu}(x)=\sum_{n=0}^{\infty}\frac{(-1)^n}{\Gamma(n+\nu+1)n!}\left(\frac{x}{2}\right)^{2n+2}
\end{equation}
for all $\nu> -1$ real orders. We denote the zeros of (67) as $x_{\nu,n}$ where $n$ is the index variable for the $n$th zero for $\nu$ order of the Bessel function. The Weierstrass product representation of (67) is
\begin{equation}\label{eq:20}
J_{\nu}(x)=\frac{1}{\Gamma(\nu+1)}\left(\frac{x}{2}\right)^{\nu}\prod_{n=1}^{\infty}\left(1-\frac{x^2}{x^2_{\nu,n}}\right)
\end{equation}
for $\nu>-1$ involving the zeros directly [1, p.370]. To find the roots of this function, we apply Theorem 1 to obtain a generalized zeta series over the Bessel zeros

\begin{equation}\label{eq:20}
\begin{aligned}
 Z_{\nu}(2m)&=-\frac{1}{2(2m-1)!}\frac{d^{(2m)}}{dx^{(2m)}}\log\left[\frac{J_{\nu}(x)}{x^{\nu}}\right]\Bigr\rvert_{x\to 0}\\
            &=\sum_{n=1}^{\infty}\frac{1}{x_{\nu,n}^{2m}},
\end{aligned}
\end{equation}
which is essentially the $2m$th derivative of $\log[J_{\nu}(x)/x^{\nu}]$ evaluated at $x=0$, where it is taken in a limiting sense $x\to 0$ as to avoid division by zero. Here we cancel the zero $x^{\nu}$ at origin as shown in previous section. This will prevent any artifacts of $x^{\nu}$ in (68) from being generated by the log-differentiation. Furthermore, the Bessel function is even; hence we consider the $2m$ values.

\lstset{language=C,deletekeywords={for,double,return},caption={PARI script for computing the first zero of (52) by formula (61).},label=DescriptiveLabel,captionpos=b}
\begin{lstlisting}[frame=single]
{
   m = 20;              \\ set limit variable
   delta = 10^(-100);   \\ take s->0 limit

   \\ compute generalized zeta series
   A = -derivnum(s=delta,log(sin(Pi*s)/(Pi*s)),2*m);
   B = 1/factorial(2*m-1);
   Z = A*B/2;

   \\ compute the first zero
   z1 = (Z)^(-1/(2*m));
   print(z1);
}
\end{lstlisting}

The first few special values of $Z_{\nu}(2m)$ for even orders are:
\begin{equation}\label{eq:20}
\begin{aligned}
Z_{\nu}(2) &= \frac{1}{4(\nu+1)},\\
\\
Z_{\nu}(4) &= \frac{1}{16(\nu+1)^2(\nu+2)},\\
\\
Z_{\nu}(6) &= \frac{1}{32(\nu+1)^3(\nu+2)(\nu+3)},\\
\\
Z_{\nu}(8) &= \frac{5\nu+11}{256(\nu+1)^4(\nu^2+2)^2(\nu+3)(\nu+4)},\\
\\
Z_{\nu}(10) &= \frac{7\nu+19}{512(\nu+1)^5(\nu^2+2)^2(\nu+3)(\nu+4)(\nu+5)},\\
\\
Z_{\nu}(12) &= \frac{21\nu^3+181\nu^2+513\nu+473}{2048(\nu+1)^6(\nu^2+2)^3(\nu+3)^2(\nu+4)(\nu+5)(\nu+6)},
\end{aligned}
\end{equation}
and so on. These generated values are rational functions of the Bessel order $\nu>-1$ for even orders, and this implies that if $\nu>-1$ is rational, so is $Z_{\nu}(2m)$. These formulas for $Z_{\nu}(2m)$ were generated using another recurrence relation found Sneddon in [18] as an alternative to (69), which is given in Appendix B. In Table 1, we give the values of $Z_{\nu}(2m)$ for different $m$ and $\nu$.

\begin{table}[hbt!]
\caption{Generated values of $Z_{\nu}(m)$ for different $m$ and $\nu$.}
\centering
\def\arraystretch{2}
\begin{tabular}{|c| c| c| c| c|} \hline
m &$Z_{0}(m)$  & $Z_{1}(m)$  & $Z_{2}(m)$  & $Z_{3}(m)$   \\ \hline
2 & \Large $\frac{1}{4}$ & \Large $\frac{1}{8}$ & \Large $\frac{1}{12}$ & \Large $\frac{1}{16}$ \\ \hline
4 & \Large $\frac{1}{32}$ & \Large $\frac{1}{192}$ & \Large $\frac{1}{576}$  & \Large $\frac{1}{1280}$ \\ \hline
6 & \Large $\frac{1}{192}$ & \Large $\frac{1}{3072}$ & \Large  $\frac{1}{17280}$  & \Large $\frac{1}{61440}$ \\ \hline
8 & \Large $\frac{11}{12288}$ & \Large $\frac{1}{46080}$ & \Large $\frac{7}{3317760}$  & \Large $\frac{13}{34406400}$ \\ \hline
10 & \Large $\frac{19}{122880}$ &\Large $\frac{13}{8847360}$ & \Large  $\frac{11}{139345920}$  & \Large $\frac{1}{110100480}$ \\ \hline
12 & \Large$\frac{473}{17694720}$ & \Large $\frac{11}{110100480}$ & \Large $\frac{797}{267544166400}$  & \Large $\frac{263}{1189085184000}$  \\ \hline
\end{tabular}
\label{table:nonlin}
\end{table}

And now, by applying Theorem 2,  we obtain a full recurrence formula satisfied by the Bessel zeros:
\begin{equation}\label{eq:20}
x_{\nu,n+1}=\lim_{s\to\infty}\left(Z_{\nu}(s)-\sum_{k=1}^{n}\frac{1}{x_{\nu,k}^{s}}\right)^{-1/s}.
\end{equation}
To verify (71) numerically, we compute the principal zero using (69), since it is more efficient than (70), for the limit variable $m=250$ and $\nu=0$ which results in
\begin{equation}\label{eq:20}
\begin{aligned}
x_{0,1} &=\lim_{m\to\infty}\left[Z_{0}(2m)\right]^{-\frac{1}{2m}}\\
          &= 2.404825557695772768621631879326\ldots
\end{aligned}
\end{equation}
which is accurate to 181 decimal places (we are just showing the first 30 digits). In [21, p.500-503], Rayleigh-Cayley generated values for $Z_{\nu}(2m)$ as shown in (70) and extended Euler's original work and computed the smallest Bessel zero using this method in papers dating back to year 1874. Moving on, the next zero is found the same way, we compute
\begin{equation}\label{eq:20}
\begin{aligned}
x_{0,2} &=\lim_{m\to\infty}\left[Z_{0}(2m)-\frac{1}{x^{2m}_{0,1}}\right]^{-\frac{1}{2m}}\\
          &= 5.5200781102863106495966041128130\ldots
\end{aligned}
\end{equation}
for $m=250$, and is accurate to $99$ decimal places, but in order to ensure convergence, the first zero $x_{0,1}$ has to be known to high enough precision (usually much higher than can be efficiently computed using this method as we did above). Henceforth, as a numerical experiment, we take $x_{0,1}$ that was already pre-computed to high enough precision using more efficient means to 1000 decimal places using the standard equation solver found on mathematical software packages (such root finding algorithms are very effective, but must assume an initial condition), rather than taking the zero computed above with less accuracy.  And similarly, the third zero is computed as
\begin{equation}\label{eq:20}
\begin{aligned}
x_{0,3} &=\lim_{m\to\infty}\left[Z_{0}(2m)-\frac{1}{x^{2m}_{0,1}}-\frac{1}{x^{2m}_{0,2}}\right]^{-\frac{1}{2m}}\\
          &=8.6537279129110122169541987126609\ldots
\end{aligned}
\end{equation}
which is accurate to $68$ decimal places and that it was assumed $x_{0,1}$ and  $x_{0,2}$ was already pre-computed to high enough precision (1000 decimal places using the standard equation solver) in order to ensure convergence. Hence in general, one can continue on and keep removing all the known zeros up to the $n$th order in order to compute the $n$th+1 zero. In numerical computations, the key here is that the accuracy of the previous zeros must be much higher than the next zero in order to ensure convergence, i.e., $x_{\nu,n}^{-s}\gg x_{\nu,n+1}^{-s}$, and also one cannot use the same limit variable to compute the next zero based on the previous zero as it will cause a self-cancelation in the formula. Numerically, there is a fine balance as to how many accurate digits are available and the magnitude of the limit variable $m$ used to compute the next zero. We also note that this method is not numerically an efficient method to compute zeros, but it allows to have a true closed-form representation of the zeros, and also that one does not need to make an initial guess for the zero, as is generally the case for many root finding algorithms.

We also remarked that the generalized zeta series will be rational for rational Bessel order $\nu>-1$. Since the first zero can be written as

\begin{equation}\label{eq:20}
x_{\nu,1}=\lim_{m\to\infty}\left[Z_{\nu}(2m)\right]^{-\frac{1}{2m}},
\end{equation}
and that implies that we have a $2m$-th root of a rational number, which is irrational. Hence for most purposes, the sequence converging to the first Bessel zero for any rational $\nu>-1$ order will be irrational up to the limit variable $m$. For example, for $Z_{\nu}(2m)$ for $m=6$ in (70), we have an approximation to converging to the first zero

\begin{equation}\label{eq:20}
x_{\nu,1}\approx\left[\frac{21\nu^3+181\nu^2+513\nu+473}{2048(\nu+1)^6(\nu^2+2)^3(\nu+3)^2(\nu+4)(\nu+5)(\nu+6)}\right]^{-\frac{1}{12}}
\end{equation}
which is irrational for any rational $\nu>-1$. We remark that this is not a definite proof of the irrationality of the Bessel zero, but rather, a condition where one can set $m$ arbitrarily high as $m\to\infty$, and the sequence converging to the first Bessel zero will be irrational.

The presented methods by Theorem 1 and Theorem 2 can be effectively used to find zeros of many different functions, such as the digamma function, the Bessel functions, the Airy function, and many other finite and infinite degree polynomials (provided that the zeros are loosely well-behaved), and in the next section, we will apply this method to find the trivial and non-trivial zeros of the Riemann zeta function.

\section{Formulas for the Riemann zeros}

As described in the Introduction, the Riemann zeta function consists of trivial zeros $\rho_t$ and non-trivial zeros $\rho_{nt}$, so that the full generalized zeta series over all zeros is

\begin{equation}\label{eq:20}
Z(s)=Z_{t}(s)+Z_{nt}(s)
\end{equation}
where
\begin{equation}\label{eq:20}
Z_{t}(s)=\sum_{n=1}^{\infty}\frac{1}{\rho^s_{t,n}}
\end{equation}
and
\begin{equation}\label{eq:20}
Z_{nt}(s)=\sum_{n=1}^{\infty}\left(\frac{1}{\rho^s_{nt,n}}+\frac{1}{\bar{\rho}^s_{nt,n}}\right)
\end{equation}
are the trivial and non-trivial components, where they are taken in conjugate-pairs. Now, when applying the root-extraction to (77) directly is not straightforward. First we observe that

\begin{equation}\label{eq:20}
O[Z_{t}(s)]\gg O[Z_{nt}(s)]
\end{equation}
as $s\to\infty$, since

\begin{equation}\label{eq:20}
\frac{1}{2^s}\gg \left|\frac{1}{(\frac{1}{2}+it_1)^s}+\frac{1}{(\frac{1}{2}-it_1)^s}\right|
\end{equation}
or roughly
\begin{equation}\label{eq:20}
\frac{1}{2^s}\gg \frac{1}{t_1^s}.
\end{equation}
Hence, $Z_{t}(s)$ dominates the limit in relation to $Z_{nt}(s)$.

Next, we develop a formula for trivial zeros using the root-extraction method. It is first convenient to remove the pole of $\zeta(s)$ by inspecting the Weierstrass infinite product (8) to consider the entire function

\begin{equation}\label{eq:20}
f(s)=\zeta(s)(s-1),
\end{equation}
then the $m$th log-derivative gives the generalized zeta series over all zeros as

\begin{equation}\label{eq:20}
Z(s)=Z_{t}(m)+Z_{nt}(m)=-\frac{1}{(m-1)!}\frac{d^{m}}{ds^{m}}\log\left[\zeta(s)(s-1)\right]\Bigr\rvert_{s\to 0}.
\end{equation}
As a result, the recurrence formula for trivial zeros is:
\begin{equation}\label{eq:20}
\begin{aligned}
\rho_{t,n+1}=-\lim_{m\to\infty}\Bigg[-\frac{1}{(2m-1)!}&\frac{d^{(2m)}}{ds^{(2m)}}\log\left[\zeta(s)(s-1)\right]\Bigr\rvert_{s\to 0}-\sum_{k=1}^{n}\frac{1}{\rho_{t,k}^{2m}}+\\
&-\sum_{k=1}^{\infty}\left(\frac{1}{\rho^{2m}_{nt,k}}+\frac{1}{\bar{\rho}^{2m}_{nt,k}}\right)\Bigg]^{-\frac{1}{2m}}.
\end{aligned}
\end{equation}
We first note that we have used a $2m$ limiting value, but in this case, one could also use an odd limit value, but then an alternating sign $(-1)^{m}$ is needed in the recurrence formula to account for positive and negative terms, but we wish to omit that. Secondly, we've also added a negative sign in front to account for a negative branch in $s$-domain restricted to $-0.5<\Re(s)<\{0.0091598901\ldots \cup \Im(s)=0\}$, so that trivial zeros will come out negative. This sign change will be more apparent in later sections. Thirdly, there is a contribution due to the conjugate-pairs of non-trivial zeros. Initially, the $Z_{t}(s)$ is the dominant lowest term; hence the contribution due to non-trivial zeros is negligible and may be dropped, but during the course of removing the trivial zeros recursively in order to generate the $n$th+1 trivial zero, we eventually arrive at a point where the first non-trivial zero term dominates the limit as we will see shortly.

First we begin to verify the trivial zero formula numerically and we compute the principal zero as

\begin{equation}\label{eq:20}
\rho_{t,1}=-\lim_{m\to\infty}\left[-\frac{1}{(2m-1)!}\frac{d^{(2m)}}{ds^{(2m)}}\log\left[\zeta(s)(s-1)\right]\Bigr\rvert_{s\to 0}\right]^{-\frac{1}{2m}},
\end{equation}
and the the script in PARI is shown in Listing 2. By running the script for a limit variable $m=20$ we compute

\begin{equation}\label{eq:20}
\rho_{t,1}=-1.999999999999\underline{9}54525260798701750 \ldots
\end{equation}
which is a close approximation to within $13$ decimal places. The next zero is recursively found as

\begin{equation}\label{eq:20}
\rho_{t,2}=-\lim_{m\to\infty}\left[-\frac{1}{(2m-1)!}\frac{d^{(2m)}}{ds^{(2m)}}\log\left[\zeta(s)(s-1)\right]\Bigr\rvert_{s\to 0}-\frac{1}{(-2)^{2m}}\right]^{-\frac{1}{2m}},
\end{equation}
but this time we must know the first zero in advance, so that we now compute
\begin{equation}\label{eq:20}
\rho_{t,2}= -3.9999999\underline{9}0956136773796946421201\ldots
\end{equation}
accurate to within 8 decimal places, and similarly, the third zero is

\begin{equation}\label{eq:20}
\rho_{t,3}=-\lim_{m\to\infty}\left[-\frac{1}{(2m-1)!}\frac{d^{(2m)}}{ds^{(2m)}}\log\left[\zeta(s)(s-1)\right]\Bigr\rvert_{s\to 0}-\frac{1}{(-2)^{2m}}-\frac{1}{(-4)^{2m}}\right]^{-\frac{1}{2m}},
\end{equation}
but this time we must know first two previous zeros, so that we compute

\begin{equation}\label{eq:20}
\rho_{t,3}= -5.9999\underline{9}8491319353326392575769396\ldots,
\end{equation}
which is accurate to 5 decimal places. We see that the accuracy progressively reduces for higher zeros, and so, the limit variable $m$ has to be increased to get better accuracy.

\lstset{language=C,deletekeywords={for,double,return},caption={PARI script for computing the first trivial zero using (86).},label=DescriptiveLabel,captionpos=b}
\begin{lstlisting}[frame=single]
{
   m = 20; \\ set limit variable

   \\ compute generalized zeta series
   A = -derivnum(s=0,log(zeta(s)*(s-1)),2*m);
   B = 1/factorial(2*m-1);
   Z = A*B;

   \\ compute the first trivial zero
   rho_t1 = Z^(-1/(2*m));
   print(rho_t1);
}
\end{lstlisting}

We keep repeating this, but now re-compute with $m=200$ to get better accuracy until we get to the $7$th trivial zero, which should be $-14$, so by computing

\begin{equation}\label{eq:20}
\begin{aligned}
\rho_{t,7}=-\lim_{m\to\infty}&\Bigg[-\frac{1}{(2m-1)!}\frac{d^{(2m)}}{ds^{(2m)}}\log\left[\zeta(s)(s-1)\right]\Bigr\rvert_{s\to 0}-\frac{1}{(-2)^{2m}}-\frac{1}{(-4)^{2m}}-\frac{1}{(-6)^{2m}}+\\
                            &-\frac{1}{(-8)^{2m}}-\frac{1}{(-10)^{2m}}-\frac{1}{(-12)^{2m}}\Bigg]^{-\frac{1}{2m}}
\end{aligned}
\end{equation}
yields
\begin{equation}\label{eq:20}
\rho_{t,7}= -14.000007669086476837019928729271\ldots,
\end{equation}
where we notice that it's becoming less accurate. So when we compute the next zero

\begin{equation}\label{eq:20}
\begin{aligned}
\rho_{t,8}=-\lim_{m\to\infty}&\Bigg[-\frac{1}{(2m-1)!}\frac{d^{(2m)}}{ds^{(2m)}}\log\left[\zeta(s)(s-1)\right]\Bigr\rvert_{s\to 0}-\frac{1}{(-2)^{2m}}-\frac{1}{(-4)^{2m}}-\frac{1}{(-6)^{2m}}+\\
                            &-\frac{1}{(-8)^{2m}}-\frac{1}{(-10)^{2m}}-\frac{1}{(-12)^{2m}}-\frac{1}{(-14)^{2m}}\Bigg]^{-\frac{1}{2m}},
\end{aligned}
\end{equation}
we get
\begin{equation}\label{eq:20}
\rho_{t,8}= 14.2975976399\ldots-i 0.1122953782\ldots,
\end{equation}
where it is seen no longer converging to $-16$ as expected. However, due to equality (82), the first non-trivial zero term is dominating the limit, so if we incorporate the first non-trivial zero term
\begin{equation}\label{eq:20}
\begin{aligned}
\rho_{t,8}=-\lim_{m\to\infty}&\Bigg[-\frac{1}{(2m-1)!}\frac{d^{(2m)}}{ds^{(2m)}}\log\left[\zeta(s)(s-1)\right]\Bigr\rvert_{s\to 0}-\frac{1}{(-2)^{2m}}-\frac{1}{(-4)^{2m}}-\frac{1}{(-6)^{2m}}+\\
                            &-\frac{1}{(-8)^{2m}}-\frac{1}{(-10)^{2m}}-\frac{1}{(-12)^{2m}}-\frac{1}{(-14)^{2m}}-\frac{1}{(\frac{1}{2}+it_1)^{2m}}-\frac{1}{(\frac{1}{2}-it_1)^{2m}}\Bigg]^{-\frac{1}{2m}}
\end{aligned}
\end{equation}
as to remove its contribution, then we re-compute the trivial zero again
\begin{equation}\label{eq:20}
\rho_{t,8}= -15.99999999999999999999\underline{9}861627109\ldots
\end{equation}
as desired. Hence this process continues for the $n$th+1 trivial zero, but for higher trivial zero terms, more non-trivial terms have to be removed in this fashion.

Moving on next, we seek to find a formula for non-trivial zeros, but according to (82), the trivial zeros dominate the generalized zeta series, and also that non-trivial zeros are complex will make the root extraction more difficult. So we consider the Weierstrass/Hadamard product (8) again but re-write it in a simpler form

\begin{equation}\label{eq:20}
\zeta(s)=\frac{\pi^{s/2}}{2(s-1)\Gamma(1+\frac{s}{2})}\prod_{\rho_{nt}}^{}\left(1-\frac{s}{\rho_{nt}}\right),
\end{equation}
as to compress the trivial zeros by the gamma function which has the Weierstrass product

\begin{equation}\label{eq:20}
\Gamma(s)=\frac{e^{-\gamma s}}{s}\prod_{n=1}^{\infty}\left(1+\frac{s}{n}\right)^{-1}e^{\frac{s}{n}}
\end{equation}
and $\Gamma(s)$ is also known to have many representations making it a useful function.
We now consider the Riemann xi function

\begin{equation}\label{eq:20}
\begin{aligned}
\xi(s) &=\frac{(s-1)\Gamma(1+\frac{s}{2})}{\pi^{s/2}}\zeta(s)\\
     &=\frac{1}{2}\prod_{\rho_{nt}}^{}\left(1-\frac{s}{\rho_{nt}}\right)
\end{aligned}
\end{equation}
as to remove all trivial zero terms (and any other remaining terms) in order to obtain an exclusive access to non-trivial zeros. Now, when applying the $m$th log-derivative to $\xi(s)$ we get

\begin{equation}\label{eq:20}
\begin{aligned}
Z_{nt}(m) &=-\frac{1}{(m-1)!}\frac{d^{m}}{ds^{m}}\log\left[\frac{(s-1)\Gamma(1+\frac{s}{2})}{\pi^{s/2}}\zeta(s)\right]\Bigr\rvert_{s\to 0}\\
           \\
          &= \sum_{n=1}^{\infty}\Bigg[\frac{1}{(\sigma_n+it_n)^{m}}+\frac{1}{(\sigma_n-it_n)^{m}}\Bigg],
\end{aligned}
\end{equation}
which is valid for $m\geq 1$. The first few special values of this series are:

\begin{equation}\label{eq:9}
\begin{aligned}
    Z_{nt}(1) &= 1-\frac{1}{2}\eta_0-\frac{1}{2}\log(4\pi)\\
    &= 1+\frac{1}{2}\gamma-\frac{1}{2}\log(4\pi)\\
    &= 0.023095708966121033814310247906\ldots, \\
    &\\
    Z_{nt}(2) &= 1+\eta_1-\frac{1}{8}\pi^2\\
    &= 1+\gamma^2+2\gamma_1-\frac{1}{8}\pi^2\\
    &=  -0.046154317295804602757107990379\dots,\\
    &\\
    Z_{nt}(3) &= 1-\eta_2-\frac{7}{8}\zeta(3)\\
    &= 1+\gamma^3+3\gamma\gamma_1+\frac{3}{2}\gamma_2-\frac{7}{8}\zeta(3)\\
    &= -0.000111158231452105922762668238\dots, \\
    &\\
    Z_{nt}(4) &= 1+\eta_3-\frac{1}{96}\pi^4\\
    &= 1+\gamma^4+4\gamma^2\gamma_1+2\gamma_1^2+2\gamma\gamma_2+\frac{2}{3}\gamma_3-\frac{1}{96}\pi^4\\
    &= 0.000073627221261689518326771307\dots,\\
    &\\
    Z_{nt}(5) &= 1-\eta_4-\frac{31}{32}\zeta(5)\\
    &= 1+\gamma^5+5\gamma^3\gamma_1+\frac{5}{2}\gamma^2\gamma_2+\frac{5}{2}\gamma_1\gamma_2+5\gamma\gamma_1^2+\frac{5}{6}\gamma\gamma_3+\frac{5}{24}\gamma_4-\frac{31}{32}\zeta(5)\\
    &= 0.000000715093355762607735801093\dots.
\end{aligned}
\end{equation}
The value for $Z_{nt}(1)$ is commonly known throughout the literature [7], and values for $Z_{nt}(m)$ for $m>1$ also have a closed-form formula

\begin{equation}\label{eq:20}
\begin{aligned}
Z_{nt}(m) = 1-(-1)^m2^{-m}\zeta(m)-\frac{\log(|\zeta|)^{(m)}(0)}{(m-1)!}
\end{aligned}
\end{equation}
valid for $m>1$ and is given by Matsuoka [16, p.249], Lehmer [15, p.23], and Voros in [20, p.73]. This formula is valid for even and odd index variable $m$. Another representation of (101) is given by

\begin{equation}\label{eq:20}
\begin{aligned}
Z_{nt}(m) = 1- (1 - 2^{-m})\zeta(m) + (-1)^m\eta_{m-1}
\end{aligned}
\end{equation}
for $m>1$ where $\eta_n$ are the Laurent expansion coefficients of the series

\begin{equation}\label{eq:20}
-\frac{\zeta'(s)}{\zeta(s)}=\frac{1}{s-1}+\sum_{n=0}^{\infty}\eta_n(s-1)^n.
\end{equation}
The first few values are:
\begin{equation}\label{eq:9}
\begin{aligned}
    \eta_0 &= -0.57721566490153286061\ldots,\\
    &\\
    \eta_1 &=  0.18754623284036522460\ldots,\\
    &\\
    \eta_2 &=  -0.051688632033192893802\ldots,\\
    &\\
    \eta_3 &=  0.014751658825453744065\ldots,\\
    &\\
    \eta_4 &=  -0.0045244778884953787412\ldots.\\
\end{aligned}
\end{equation}
These eta constants are probably less familiar than the Stieltjes constants $\gamma_n$, and one has $-\eta_0=\gamma_0=\gamma$, but more about its relation to Stieltjes constants a little later, as our immediate goal is to express concisely

\begin{equation}\label{eq:20}
\eta_n=\frac{(-1)^{n}}{n!}\lim_{k\to\infty}\Bigg\{\sum_{l=1}^{k}\Lambda(l)\frac{\log^n(l)}{l}-\frac{\log^{n+1}(k)}{n+1}\Bigg\}
\end{equation}
as found in [20,p.25], where the von Mangoldt's function is defined as

\begin{equation}
\Lambda(n)= \left \{
\begin{aligned}
&\log p, &&\text{if}\ n=p^k \text{ for some prime and integer } k\geq 1 \\
&0 && \text{otherwise},
\end{aligned} \right.
\end{equation}
which is purely in terms of primes. Hence the expansion coefficients $\eta_n$ are written as a function of primes, and then it follows that the generalized zeta series $Z_{nt}(m)$ can also be represented in terms of primes. We note, however, that the limit identity (107) is extremely slow to converge, requiring billions of prime terms to compute to only a few digits, making it very impractical.

Continuing on, we note that the generalized zeta series $Z_{nt}(s)$ is over all zeros, but in order to extract the non-trivial zeros, we will have to assume (RH) in order to remove the real part of $\frac{1}{2}$. It is not readily possible to separate the reciprocal of conjugate-pairs of non-trivial zeros from (101), but what we can do is to consider a secondary zeta function over the complex magnitude, or modulus, squared of non-trivial zeros as

\begin{equation}\label{eq:20}
\begin{aligned}
Z_{|nt|}(s) &= \sum_{n=1}^{\infty}\frac{1}{|\rho_{nt}|^{2s}}=\sum_{n=1}^{\infty}\frac{1}{(\frac{1}{4}+t_n^2)^{s}}
\end{aligned}
\end{equation}
and then by applying Theorem 2, we can obtain non-trivial zeros

\begin{equation}\label{eq:20}
t_{n+1}=\lim_{m\to\infty}\left[\left(Z_{|nt|}(m)-\sum_{k=1}^{n}\frac{1}{(\frac{1}{4}+t_k^2)^{m}}\right)^{-1/m}-\frac{1}{4}\right]^{1/2}
\end{equation}
recursively. We now need a closed-form formula for $Z_{|nt|}(m)$ which we find can be related to $Z_{nt}(m)$ in several ways. The first way is by an asymptotic formula

\begin{equation}\label{eq:20}
Z_{|nt|}(m) \sim \frac{1}{2}[Z_{nt}^2(m)-Z_{nt}(2m)]
\end{equation}
as $m\to \infty$. This immediately leads to a formula for the principal zero

\begin{equation}\label{eq:20}
t_{1}=\lim_{m\to\infty}\left[\left(\frac{1}{2}Z_{nt}^2(m)-\frac{1}{2}Z_{nt}(2m)\right)^{-1/m}-\frac{1}{4}\right]^{1/2},
\end{equation}
and a full recurrence formula

\begin{equation}\label{eq:20}
t_{n+1}=\lim_{m\to\infty}\left[\left(\frac{1}{2}Z_{nt}^2(m)-\frac{1}{2}Z_{nt}(2m)-\sum_{k=1}^{n}\frac{1}{(\frac{1}{4}+t_k^2)^{m}}\right)^{-1/m}-\frac{1}{4}\right]^{1/2}
\end{equation}
for non-trivial zeros as we have shown in [13, p.9-14] and Matsuoka in [16], provided that all the zeros are assumed to lie on the critical line. A detailed numerical computation of $t_1$ by equation (112) is shown in Table 2 and a script in PARI in Listing 3, where we can observe convergence to $t_1$ as the limit variable $m$ increases from low to high, and at $m=100$ we get over $16$ decimal places. A detailed numerical computation for higher $m$ is summarized in [13].

\begin{table}[hbt!]
\caption{The computation of $t_1$ by equation (112) for different $m$.} 
\centering 
\begin{tabular}{c c c} 
\hline\hline 
m & $t_1$ (First 30 Digits)  & Significant Digits\\ [0.5ex] 
\hline 
$2$ & 5.561891787634141032446012810136 & 0 \\
$3$ & 13.757670503723662711511861003244 & 0 \\
$4$ & 12.161258748655529488677538477512 & 0 \\
$5$ & 14.075935317783371421926582853327 & 0 \\
$6$ & 13.579175424560852302300158195372 & 0 \\
$7$ & 14.\underline{1}16625853057249358432588137893 & 1 \\
$8$ & 13.961182494234115467191058505224 & 0 \\
$9$ & 14.\underline{1}26913415083941105873032355837 & 1 \\
$10$ &14.077114859427980275510456957007 & 0 \\
$15$ &14.1\underline{3}3795710050725394699252528681 & 2 \\
$20$ &14.13\underline{4}370485636531946259958638820 & 3 \\
$25$ &14.134\underline{7}00629574414322701677282886 & 4 \\
$50$ &14.13472514\underline{1}835685792188021492482 & 9 \\
$100$&14.134725141734693\underline{7}89329888107217 & 16
\\ [1ex] 
\hline 
\end{tabular}
\label{table:nonlin} 
\end{table}

\newpage

\lstset{language=C,deletekeywords={for,double},caption={PARI script for computing equation (112).},label=DescriptiveLabel,captionpos=b}
\begin{lstlisting}[frame=single]
{
    m1 = 250;  \\ set limit variables
    m2 = 2*m1;

    \\ compute parameters A1 to C1 for Z1
    A1 =  derivnum(x=0,log(abs(zeta(x))),m1);
    B1 = 1/factorial(m1-1);
    C1 = 1-(-1)^m1*2^(-m1)*zeta(m1);
    Z1 = C1-A1*B1;

    \\ compute parameters A2 to C2 for Z2
    A2 =  derivnum(x=0,log(abs(zeta(x))),m2);
    B2 = 1/factorial(m2-1);
    C2 = 1-(-1)^m2*2^(-m2)*zeta(m2);
    Z2 = C2-A2*B2;

    t1 = (((Z1^2-Z2)/2)^(-1/m1)-1/4)^(1/2);
    print(t1);
}
\end{lstlisting}

\noindent The next higher order zeros are recursively found as

\begin{equation}\label{eq:20}
t_{2}=\lim_{m\to\infty}\left[\left(\frac{1}{2}Z_{nt}^2(m)-\frac{1}{2}Z_{nt}(2m)-\frac{1}{(\frac{1}{4}+t_1^2)^{m}}\right)^{-1/m}-\frac{1}{4}\right]^{\frac{1}{2}},
\end{equation}
and the next is

\begin{equation}\label{eq:20}
t_{3}=\lim_{m\to\infty}\left[\left(\frac{1}{2}Z_{nt}^2(m)-\frac{1}{2}Z_{nt}(2m)-\frac{1}{(\frac{1}{4}+t_1^2)^{m}}-\frac{1}{(\frac{1}{4}+t_2^2)^{m}}\right)^{-1/m}-\frac{1}{4}\right]^{\frac{1}{2}},
\end{equation}
and so on, but the numerical computation is even more difficult, and so, the limit variable $m$ has to be increased to a very large value.
We can now express the non-trivial zeros in terms of other constants. By substituting the eta constants (104) to (112), we obtain the first zero:

\begin{equation}\label{eq:20}
\begin{aligned}
t_{1} =\lim_{m\to\infty}&\Bigg[\Bigg(\frac{1}{2}\Big(1- (1 - 2^{-m})\zeta(m) + (-1)^m\eta_{m-1}\Big)^2+\\
         & -\frac{1}{2}\Big(1- (1 - 2^{-2m})\zeta(2m) + \eta_{2m-1}\Big)\Bigg)^{-1/m}-\frac{1}{4}\Bigg]^{\frac{1}{2}}.
\end{aligned}
\end{equation}

\noindent For example, if we let $m=10$ then we can generate an approximation converging to $t_1$ as

\begin{equation}\label{eq:20}
\begin{aligned}
t_{1} &\approx \Bigg[\Bigg(-\frac{31}{2903040}\pi^{10}(1+\eta_9)+\eta_9+\frac{1}{2}\eta^2_9-\frac{1}{2}\eta_{19}+\frac{10568303}{92681981263872000}\pi^{20}\Bigg)^{-\frac{1}{10}}-\frac{1}{4}\Bigg]^{\frac{1}{2}}\\
\\
&\approx 14.07711485942798027551\ldots
\end{aligned}
\end{equation}
where it is seen converging to $t_1$. For this computation, we compute the eta constants:
\begin{equation}\label{eq:20}
\begin{aligned}
\eta_{9} &=  0.000017041357047110641032\ldots,\\
\\
\eta_{19} &= 0.000000000286807697455596\ldots
\end{aligned}
\end{equation}
using $m$th differentiation of (105). As mentioned before, one could alternatively compute these eta constants using primes by (107), but the number of primes required now would be in trillions (making it very impractical to compute on a standard workstation), but the main point is that the non-trivial zeros can be expressed in terms of primes, namely, by equations (107), (105) and (112).

Furthermore, a recurrence relation for the eta constants in terms of Stieltjes constants is
\begin{equation}\label{eq:20}
\eta_n=(-1)^{n+1}\left[\frac{n+1}{n!}\gamma_n+\sum_{k=0}^{n-1}\frac{(-1)^{k-1}}{(n-k-1)!}\eta_k\gamma_{n-k-1}\right]
\end{equation}
found in Coffey [6, p.532]. Using these relations, the non-trivial zeros can be written in terms of Stieltjes constants. For the first zero $t_1$ and $m=2$, we obtain an expansion:

\begin{equation}\label{eq:20}
\begin{aligned}
t_1 &\approx \left[\left(2\gamma_1-\frac{\pi^2\gamma_1}{4}+\gamma_1^2-\gamma\gamma_2-\frac{\gamma_3}{3}+\gamma^2-\frac{\pi^2}{8}-\frac{\gamma^2\pi^2}{8}+\frac{5\pi^4}{384}\right)^{-\frac{1}{2}}-\frac{1}{4}\right]^{\frac{1}{2}}
\end{aligned}
\end{equation}
\begin{equation}\label{eq:20}
t_1\approx 5.561891787634141032446012810136.\ldots\nonumber
\end{equation}
For $m=3$ we obtain an expansion:

\begin{equation}\label{eq:20}
\begin{aligned}
t_1 \approx & \Bigg[\Bigg(\gamma^3-\frac{21}{8}\gamma\gamma_1\zeta(3)-\frac{21}{16}\gamma_2\zeta(3)+3\gamma\gamma_1-\gamma_1^3+\frac{3}{2}\gamma_2+\frac{3}{2}\gamma\gamma_1\gamma_2+\\
&+\frac{3}{4}\gamma_2^2-\frac{1}{2}\gamma^2\gamma_3-\frac{1}{2}\gamma_1\gamma_3-\frac{1}{8}\gamma\gamma_4-\frac{1}{40}\gamma_5-\frac{7}{8}\zeta(3)-\frac{7}{8}\gamma^3\zeta(3)+\\
&+\frac{49}{128}\zeta(3)^2+\frac{1}{1920}\pi^6\Bigg)^{-\frac{1}{3}}-\frac{1}{4}\Bigg]^{\frac{1}{2}}
\end{aligned}
\end{equation}

\begin{equation}\label{eq:20}
t_1\approx 13.757670503723662711511861003244\ldots.\nonumber
\end{equation}
For $m=4$ we obtain an expansion:

\begin{equation}\label{eq:20}
\begin{aligned}
t_1 \approx & \Bigg[\Bigg(4\gamma^2\gamma_1-\frac{1}{24}\gamma^2\gamma_1\pi^4+2\gamma_1^2-\frac{1}{48}\pi^4\gamma_1^2+\gamma_1^4+2\gamma\gamma_2-\frac{1}{48}\gamma\gamma_2\pi^4+\\
&-2\gamma\gamma_1^2\gamma_2+\frac{1}{2}\gamma^2\gamma_2^2-\gamma_1\gamma_2^2+\frac{2}{3}\gamma_3-\frac{1}{144}\pi^4\gamma_3+\frac{2}{3}\gamma^2\gamma_1\gamma_3+\frac{2}{3}\gamma_1^2\gamma_3+\\
& +\frac{2}{3}\gamma\gamma_2\gamma_3+\frac{1}{6}\gamma_3^2-\frac{1}{6}\gamma^3\gamma_4-\frac{1}{3}\gamma\gamma_1\gamma_4-\frac{1}{12}\gamma_2\gamma_4-\frac{1}{30}\gamma^2\gamma_5-\frac{1}{180}\gamma\gamma_6+\\
&-\frac{1}{1260}\gamma_7+\gamma^4-\frac{\pi^4}{96}-\frac{1}{96}\pi^4\gamma^4+\frac{23}{215040}\pi^8\Bigg)^{-\frac{1}{4}}-\frac{1}{4}\Bigg]^{\frac{1}{2}}
\end{aligned}
\end{equation}

\begin{equation}\label{eq:20}
t_1\approx 12.161258748655529488677538477512\ldots.\nonumber
\end{equation}
Hence, as $m$ increases, the value converges to $t_1$ as shown in Table 2, but the number of Stieltjes constants terms grows very large. In Table 2, we see that the accuracy of $t_1$ for odd $m$ is slightly better than for even $m+1$. We recall that the Stieltjes constants $\gamma_n$ themselves are defined as the Laurent expansion coefficients of the Riemann zeta function about $s=1$ as

\begin{equation}\label{eq:20}
\zeta(s)=\frac{1}{s-1}+\sum_{n=0}^{\infty}(-1)^n\frac{\gamma_n(s-1)^n}{n!}
\end{equation}
where they are similarly expressed as

\begin{equation}\label{eq:20}
\gamma_n=\lim_{k\to\infty}\Bigg\{\sum_{l=1}^{k}\frac{\log^n(l)}{l}-\frac{\log^{n+1}(k)}{n+1}\Bigg\}.
\end{equation}
Also, the $\gamma_0=\gamma$ is the usual Euler-Mascheroni constant.

There is also another way to compute Stieltjes constants that we developed (which can be ultimately expressed in terms of primes). We observe that $\gamma_n$ are linear coefficients in the Laurent series (123), hence if we form a system of linear equations as

\begin{gather}
 \begin{pmatrix}1 & -\frac{(s_1-1)}{1!} & \frac{(s_1-1)^2}{2!} & -\frac{(s_1-1)^3}{3!} &\dots & \frac{(s_1-1)^k}{k!}\\ 1 & -\frac{(s_2-1)}{1!} & \frac{(s_2-1)^2}{2!} & -\frac{(s_2-1)^3}{3!} & \dots & \frac{(s_2-1)^k}{k!} \\ 1 & -\frac{(s_3-1)}{1!} & \frac{(s_3-1)^2}{2!} & -\frac{(s_3-1)^3}{3!} & \dots & \frac{(s_3-1)^k}{k!}\\ \vdots & \vdots & \vdots & \vdots & \ddots  & \vdots\\ 1 & -\frac{(s_k-1)}{1!} & \frac{(s_k-1)^2}{2!} & -\frac{(s_k-1)^3}{3!} &\dots & \frac{(s_k-1)^k}{k!}\end{pmatrix}
 \begin{pmatrix}
 \gamma_0 \\
 \gamma_1 \\
 \gamma_2 \\
 \vdots \\
 \gamma_k \\
   \end{pmatrix}=
  \begin{pmatrix}
  \zeta(s_1)-\frac{1}{s_1-1}\\
  \zeta(s_2)-\frac{1}{s_2-1}\\
  \zeta(s_3)-\frac{1}{s_3-1}\\
  \vdots \\
  \ \zeta(s_k)-\frac{1}{s_k-1}\\
   \end{pmatrix},
\end{gather}
then for a choice of values for $s_1=2$, $s_2=3$, $s_3=4$ and so on, and using the Cramer's rule (for solving a system of linear equations) and some properties of an Vandermonde matrix, we find that Stieltjes constants can be represented by determinant of a certain matrix:

\begin{equation}\label{eq:20}
\gamma_n = \pm\frac{\det(A_{n+1})}{\det(A_{})}
\end{equation}
where the matrix $A_n(k)$ is matrix $A(k)$, but with an $n$th column swapped with a vector $B$ as given next

 \begin{gather}
A(k)= \begin{pmatrix}1 & -\frac{1}{1!} & \frac{1^2}{2!} & -\frac{1^3}{3!} &\dots & \frac{(-1)^k}{k!}\\ 1 & -\frac{2}{1!} & \frac{2^2}{2!} & -\frac{2^3}{3!} & \dots & \frac{(-2)^k}{k!} \\ 1 & -\frac{3}{1!} & \frac{3^2}{2!} & -\frac{3^3}{3!} & \dots & \frac{(-3)^k}{k!}\\ \vdots & \vdots & \vdots & \vdots & \ddots  & \vdots\\ 1 & -\frac{(k+1)}{1!} & \frac{(k+1)^2}{2!} & -\frac{(k+1)^3}{3!} &\dots & \frac{(-1)^k(k+1)^k}{k!}\end{pmatrix}
\end{gather}
and

\begin{gather}
B(k)=
  \begin{pmatrix}
  \zeta(2)-1\\
  \zeta(3)-\frac{1}{2}\\
  \zeta(4)-\frac{1}{3}\\
  \vdots \\
  \ \zeta(k+1)-\frac{1}{k}\\
   \end{pmatrix}.
\end{gather}
The $\pm$ sign depends on the size of the matrix $k$, but in order to ensure a positive sign, the size of $k$ must be a multiple of $4$. It can be shown that $\det(A)\to 1$, hence the determinant formula for Stieltjes constants becomes

\begin{equation}\label{eq:20}
\gamma_n = \det(A_{n+1})
\end{equation}
and the size of the matrix must be $4k$.

Hence, the first few Stieltjes constants can be represented as:

\begin{align}           
\gamma_0
&=                    
\lim_{k\to\infty}\det\begin{pmatrix} \zeta(2)-1 & -\frac{1}{1!} & \frac{1^2}{2!} & -\frac{1^3}{3!} &\dots & \frac{(-1)^k1^k}{k!}\\ \zeta(3)-\frac{1}{2} & -\frac{2}{1!} & \frac{2^2}{2!} & -\frac{2^3}{3!} & \dots & \frac{(-1)^k2^k}{k!} \\  \zeta(4)-\frac{1}{3} & -\frac{3}{1!} & \frac{3^2}{2!} & -\frac{3^3}{3!} & \dots & \frac{(-1)^k3^k}{k!}\\ \vdots & \vdots & \vdots & \vdots & \ddots  & \vdots\\ \zeta(k+1)-\frac{1}{k} & -\frac{(k+1)}{1!} & \frac{(k+1)^2}{2!} & -\frac{(k+1)^3}{3!} &\dots & \frac{(-1)^k(k+1)^k}{k!}\end{pmatrix}
\notag \\ 
\\
&\quad = 0.57721566490153286061\ldots\nonumber,
\end{align}
and the next Stieltjes constant is

\begin{align}           
\gamma_1
&=                    
\lim_{k\to\infty}\det\begin{pmatrix} 1 & \zeta(2)-1 & \frac{1^2}{2!} & -\frac{1^3}{3!} &\dots & \frac{(-1)^k1^k}{k!}\\ 1 & \zeta(3)-\frac{1}{2} & \frac{2^2}{2!} & -\frac{2^3}{3!} & \dots & \frac{(-1)^k2^k}{k!} \\  1 & \zeta(4)-\frac{1}{3} & \frac{3^2}{2!} & -\frac{3^3}{3!} & \dots & \frac{(-1)^k3^k}{k!}\\ \vdots & \vdots & \vdots & \vdots & \ddots  & \vdots\\ 1 & \zeta(k+1)-\frac{1}{k}  & \frac{(k+1)^2}{2!} & -\frac{(k+1)^3}{3!} &\dots & \frac{(-1)^k(k+1)^k}{k!}\end{pmatrix}
\notag \\ 
\\
&\quad = -0.072815845483676724861\ldots\nonumber,
\end{align}
and the next is

\begin{align}           
\gamma_2
&=                    
\lim_{k\to\infty}\det\begin{pmatrix}1 & -\frac{1}{1!} &  \zeta(2)-1 & -\frac{1^3}{3!} &\dots & \frac{(-1)^k1^k}{k!}\\ 1 & -\frac{2}{1!} & \zeta(3)-\frac{1}{2} & -\frac{2^3}{3!} & \dots & \frac{(-1)^k2^k}{k!} \\  1 & -\frac{3}{1!} & \zeta(4)-\frac{1}{3}& -\frac{3^3}{3!} & \dots & \frac{(-1)^k3^k}{k!}\\ \vdots & \vdots & \vdots & \vdots & \ddots  & \vdots\\ 1 & -\frac{(k+1)}{1!} & \zeta(k+1)-\frac{1}{k} & -\frac{(k+1)^3}{3!} &\dots & \frac{(-1)^k(k+1)^k}{k!}\end{pmatrix}
\notag \\ 
\\
&\quad = -0.0096903631928723184845\ldots\nonumber,
\end{align}
and so on. And in [12] we performed extensive numerical computation of (129) where it clearly converges to the Stieltjes constants, and in essence, analytically extends $\zeta(s)$ to the whole complex plane by the Laurent expansion (123) with an only knowledge of $\zeta(s)$ for $s>1$. Finally, we remark that the vector $B(k)$ can be expressed in terms of primes by substituting the Euler product for $\zeta(s)$ as

\begin{gather}
B(k)=
  \begin{pmatrix}
   \prod_{n=1}^{\infty}\left(1-\frac{1}{p_n^2}\right)^{-1}-1\\
   \prod_{n=1}^{\infty}\left(1-\frac{1}{p_n^3}\right)^{-1}-\frac{1}{2}\\
   \prod_{n=1}^{\infty}\left(1-\frac{1}{p_n^4}\right)^{-1}-\frac{1}{3}\\
  \vdots \\
  \  \prod_{n=1}^{\infty}\left(1-\frac{1}{p_n^{k+1}}\right)^{-1}-\frac{1}{k}\\
   \end{pmatrix}.
\end{gather}
This in turn leads to computing Stieltjes constants by primes, then $Z_{|nt|}$ by the Stieltjes constants, and then the non-trivial zeros by $Z_{|nt|}$. Henceforth, we also obtain a similar formula for the $\eta_n$ constants by defining a similar vector

\begin{gather}
D(k)=
  \begin{pmatrix}
  -\frac{\zeta^{\prime}(2)}{\zeta(2)}-1\\
  -\frac{\zeta^{\prime}(3)}{\zeta(3)}-\frac{1}{2}\\
   -\frac{\zeta^{\prime}(4)}{\zeta(4)}-\frac{1}{3}\\
  \vdots \\
  \  -\frac{\zeta^{\prime}(k+1)}{\zeta(k+1)}-\frac{1}{k} \\
   \end{pmatrix}
\end{gather}
and matrix $C_n$ which is matrix $A$ but with an $n$th column swapped with a vector $D$, and using the Cramers rule, we find that

\begin{equation}\label{eq:20}
\eta_n = \frac{1}{n!}\det(C_{n+1})
\end{equation}
and the size of the matrix must be $4k$ to ensure a positive sign.
For example, $\eta_2$ would be

\begin{align}           
\eta_2
&=                    
\lim_{k\to\infty}\frac{1}{2!}\det\begin{pmatrix}1 & -\frac{1}{1!} &  -\frac{\zeta^{\prime}(2)}{\zeta(2)}-1 & -\frac{1^3}{3!} &\dots & \frac{(-1)^k1^k}{k!}\\ 1 & -\frac{2}{1!} & -\frac{\zeta^{\prime}(3)}{\zeta(3)}-\frac{1}{2} & -\frac{2^3}{3!} & \dots & \frac{(-1)^k2^k}{k!} \\  1 & -\frac{3}{1!} & -\frac{\zeta^{\prime}(4)}{\zeta(4)}-\frac{1}{3}& -\frac{3^3}{3!} & \dots & \frac{(-1)^k3^k}{k!}\\ \vdots & \vdots & \vdots & \vdots & \ddots  & \vdots\\ 1 & -\frac{(k+1)}{1!} & -\frac{\zeta^{\prime}(k+1)}{\zeta{(k+1)}}-\frac{1}{k} & -\frac{(k+1)^3}{3!} &\dots & \frac{(-1)^k(k+1)^k}{k!}\end{pmatrix}
\notag \\ 
\\
&\quad = -0.051688632033192893802\ldots\nonumber.
\end{align}

We remarked earlier that the presented results so far are geared toward computing $Z_{|nt|}$ from $Z_{nt}$ by the asymptotic formula (111). We now investigate two other similar formulas of obtaining $Z_{|nt|}$ given by works of Voros [20]. The first is the $\mathit{Bologna}$ formula (family) as
\begin{equation}\label{eq:20}
Z_{|nt|}(m)=\sum_{n=1}^{m}\binom{2m-n-1}{m-1}Z_{nt}(n)
\end{equation}
for $m>1$ (and $t=\frac{1}{2}$) given in [20, p. 84]. And the second (very similar formula) relating to the $Z_{|nt|}$ to the Keiper-Li constants $\lambda_m$ as

\begin{equation}\label{eq:20}
Z_{|nt|}(m)=\sum_{n=1}^{m}(-1)^{n+1}\binom{2m}{m-n}\lambda_n
\end{equation}
where $\lambda_m$ are defined by

\begin{equation}\label{eq:20}
\lambda_m=\sum_{\rho_{nt}}\left[1-\left(1-\frac{1}{\rho_{nt}}\right)^m\right]
\end{equation}
which has a closed-form representation as

\begin{equation}\label{eq:20}
\lambda_m=\frac{1}{(m-1)!}\frac{d^m}{ds^m}\left[s^{m-1}\log \xi(s)\right]_{s\to 1}
\end{equation}
in terms of logarithmic differentiation of the Riemann xi function, which essentially resembles all our previous results. The first few constants are:
\begin{equation}\label{eq:9}
\begin{aligned}
    \lambda_1 &=  0.02309570896612103381\ldots,\\
    &\\
    \lambda_2 &=  0.09234573522804667038\ldots,\\
    &\\
    \lambda_3 &=  0.20763892055432480379\ldots,\\
    &\\
    \lambda_4 &=  0.36879047949224163859\ldots,\\
    &\\
    \lambda_5 &=  0.57554271446117745243\ldots,\\
\end{aligned}
\end{equation}
and so on. The Li's Criterion for (RH) is if $\lambda_m>0$ for all $m\geq 1$, which has been widely studied.  Henceforth, the first few special values of $Z_{|nt|}(m)$ in terms of Keiper-Li constants are:

\begin{equation}\label{eq:9}
\begin{aligned}
    Z_{|nt|}(1) &= \lambda_1 \\
    &= 0.023095708966121033814310247906\ldots, \\
    &\\
    Z_{|nt|}(2) &= 4\lambda_1-\lambda_2\\
    &= 0.000037100636437464871512505433\dots,\\
    &\\
    Z_{|nt|}(3) &=  15\lambda_1-6\lambda_2+\lambda_3\\
    &= 0.000000143677860288691774848062\dots, \\
    &\\
    Z_{|nt|}(4) &= 56\lambda_1-28\lambda_2+8\lambda_3-\lambda_4\\
    &= 0.000000000659827914542401152690\dots,\\
    &\\
    Z_{|nt|}(5) &= 210\lambda_1-120\lambda_2+45\lambda_3-10\lambda_4+\lambda_5\\
    &= 0.000000000003193891860867324232\dots.
\end{aligned}
\end{equation}
Now, when using the previous results, we can also compute non-trivial zeros in terms of the Keiper-Li constants. For example, for $m=10$ we approximate $t_1$ as

\begin{equation}\label{eq:20}
\begin{aligned}
t_{1}  \approx \Bigg[\Bigg(167960&\lambda_1-125970\lambda_2+77520\lambda_3-38760\lambda_4+15504\lambda_5-4845\lambda_6+\\
                   &+1140\lambda_7-190\lambda_8+20\lambda_9-\lambda_{10}\Bigg)^{-\frac{1}{10}}-\frac{1}{4}\Bigg]^{1/2}\\
\end{aligned}
\end{equation}

\begin{equation}\label{eq:20}
t_{1}  \approx 14.07711485942798027551\ldots \notag
\end{equation}
by substituting (138) to (111). A numerical computation clearly converges to the correct value as $m\to \infty$.
These formulas, together with our previous results, can be used to compute non-trivial zeros and generate a wide variety of representations of non-trivial zeros.

Moving on, another way to obtain the non-trivial zeros is to consider the secondary zeta function

\begin{equation}\label{eq:20}
Z_{1}(s)=\sum_{n=1}^{\infty}\frac{1}{t_n^{s}}
\end{equation}
over just the imaginary part of non-trivial zeros $t_n$ and applying Theorem 2 directly, where it now suffices to find a closed-form representation of $Z_{1}(s)$. To do this, we consider the Riemann xi function again, but this time transform the variable $s=\frac{1}{2}+it$ along the critical line yielding a function $\Xi(t)=\xi(\frac{1}{2}+it)$, so that its zeros are only the imaginary parts of non-trivial zeros $t_n$. Now, when applying the $m$th log-derivative formula we get

\begin{equation}\label{eq:20}
\begin{aligned}
Z_1(2m) &=-\frac{1}{2(2m-1)!}\frac{d^{(2m)}}{dt^{(2m)}}\log \Xi(t)\Bigr\rvert_{t\to 0}\\
           \\
          &= \sum_{n=1}^{\infty}\frac{1}{t_n^{2m}}
\end{aligned}
\end{equation}
for $m\geq 1$, which yields the generalized zeta series over imaginary parts of non-trivial zeros $t_n$. We note that since $\Xi(t)$ is even, we only consider the $2m$ limiting value and require a factor of $\frac{1}{2}$. The first few special values of this series are:

\begin{equation}\label{eq:9}
\begin{aligned}
Z_{1}(1) &\sim \sum_{0<t\leq T}\frac{1}{t}\sim H+\frac{1}{4\pi}\log^2\left(\frac{T}{2\pi}\right)\quad (T\to \infty) \\
         & H = -0.0171594043070981495\ldots, \\
\\
Z_{1}(2) &=\frac{1}{2}(\log |\zeta|)^{(2)}\big(\frac{1}{2}\big)+\frac{1}{8}\pi^2+\beta(2)-4 \\
     &= 0.023104993115418970788933810430\dots, \\
     &\\
Z_{1}(3) &= 0.000729548272709704215875518569\dots, \\
     &\\
Z_{1}(4) &=-\frac{1}{12}(\log |\zeta|)^{(4)}\big(\frac{1}{2}\big)-\frac{1}{24}\pi^4-4\beta(4)+16 \\
     &= 0.000037172599285269686164866262\dots, \\
     &\\
Z_{1}(5) &= 0.000002231188699502103328640628\dots.
\end{aligned}
\end{equation}
For $s=1$, the series diverges asymptotically as $H+\frac{1}{4\pi}\log^2(\frac{T}{2\pi})$ where $H$ is a constant as shown above, which is investigated by Hassani [10] and R.P. Brent [4][5], but its precise computation is very challenging because of a very slow convergence of the series, and the presented value was accurately computed to high precision by R.P. Brent [5, p.6] using $10^{10}$ non-trivial zeros and remainder estimation techniques, which further improve accuracy to over 19 decimal places. We also remark that the number of non-trivial zeros are to be taken less than or equal to $T$. The Hassani constant that results is analogous to the harmonic sum and Euler's constant relation

\begin{equation}\label{eq:20}
\sum_{n=1}^{k}\frac{1}{n}\sim \gamma+\log(k) \quad (k\to\infty).
\end{equation}
The even values of (146) given were computed using the Voros's closed-form formula

\begin{equation}\label{eq:20}
\begin{aligned}
Z_{1}(2m) = (-1)^m \bigg[-\frac{1}{2(2m-1)!}(\log |\zeta|)^{(2m)}\big(\frac{1}{2}\big)+\\
         -\frac{1}{4}\left[(2^{2m}-1)\zeta(2m)+2^{2m}\beta(2m)\right]+2^{2m}\bigg]
\end{aligned}
\end{equation}
assuming (RH). There is no known formula such as this valid for a positive odd integer argument, the odd values given were computed by an algorithm developed by Arias De Reyna [3] in a Python software package in a library \textbf{mpmath}, which roughly works by computing (144) up to several zeros and estimating the remainder to a high degree of accuracy. It would otherwise take billions of non-trivial zeros to compute (144) directly. Also, the function

\begin{equation}\label{eq:20}
\beta(s)= \sum_{n=0}^{\infty}\frac{(-1)^n}{(2n+1)^s}
\end{equation}
is the Dirichlet beta function. Finally, when applying the root-extraction to (144) by Theorem 2,  we find the principal zero as

\begin{equation}\label{eq:20}
t_{1} = \lim_{m\to\infty}\left[Z_{1}(2m)\right]^{-\frac{1}{2m}},
\end{equation}
and a numerical computation for $m=250$ yields

\begin{equation}\label{eq:20}
\begin{aligned}
t_1=14.13472514173469379045725198356247027078425711569924 & \\
     317568556746014996342980925676494901\underline{0}212214333747\ldots.
\end{aligned}
\end{equation}
using a script in Listing 4, which is accurate to $87$ decimal places. The second zero is recursively found as
\begin{equation}\label{eq:20}
t_{2} = \lim_{m\to\infty}\left[Z_{1}(2m)-\frac{1}{t_1^{2m}}\right]^{-\frac{1}{2m}}
\end{equation}
and a numerical computation for $m=250$ yields
\begin{equation}\label{eq:20}
t_2=21.0220396387715549926284795938969027773\underline{3}355195796311\ldots
\end{equation}
which is accurate to $38$ decimal places, but the first zero $t_1$ used was already pre-computed to $1000$ decimal places by other means in order to ensure convergence. We cannot substitute the same $t_1$ computed in (151) for $m=250$ to (152) as it will cause self-cancelation, and so, the accuracy of $t_{n}$ must be much higher than $t_{n+1}$. And, similarly, the third zero is recursively found as
\begin{equation}\label{eq:20}
t_{3} = \lim_{m\to\infty}\left[Z_{1}(2m)-\frac{1}{t_1^{2m}}-\frac{1}{t_2^{2m}}\right]^{-\frac{1}{2m}}
\end{equation}
and a numerical computation for $m=250$ yields
\begin{equation}\label{eq:20}
t_3=25.010857580145688763213790992562821818659549\underline{6}5846378\ldots
\end{equation}
which is accurate to $43$ decimal places, but the $t_1$ and $t_2$ zeros used were already pre-computed to $1000$ decimal places by other means in order to ensure convergence. We cannot substitute the same $t_1$ and $t_2$ computed in (151) and (153) for $m=250$ to (154) as it will cause self-cancelation, and so the accuracy of $t_{n}$ must be much higher than $t_{n+1}$. Hence, a full recurrence formula is

\begin{equation}\label{eq:20}
t_{n+1} = \lim_{m\to\infty}\left[Z_{1}(2m)-\sum_{k=1}^{n}\frac{1}{t_{k}^{2m}}\right]^{-\frac{1}{2m}}.
\end{equation}

\lstset{language=C,deletekeywords={for,double,return},caption={PARI script for computing first the non-trivial zero using equation (145) and (150).},label=DescriptiveLabel,captionpos=b}
\begin{lstlisting}[frame=single]
\\ define xi(s)
xi(s)=
{
  (s-1)*gamma(1+s/2)/Pi^(s/2)*zeta(s)
}

{
   \\ set limit variable
   m = 20;

   \\ compute generalized zeta series
   A = -derivnum(t=0,log(xi(1/2+I*t)),2*m);
   B = 1/factorial(2*m-1);

   Z = 1/2*A*B;

   \\ compute the first zero
   t1 = Z^(-1/(2*m));
   print(t1);
}
\end{lstlisting}

Furthermore, we also have a useful identity

\begin{equation}\label{eq:20}
\frac{1}{2^s}\zeta\big(s,\frac{5}{4}\big)=\sum_{k=1}^{\infty}\frac{1}{\left(\frac{1}{2}+2k\right)^s}=2^s\left[\frac{1}{2}\left((1-2^{-s})\zeta(s)+\beta(s)\right)-1\right],
\end{equation}
found in [19, p.681] for which we can express the zeta and beta terms in terms of a Hurwitz zeta function, and then substituting the Voros's closed-form formula (148) into (157) we obtain another formula for non-trivial zeros

\begin{equation}\label{eq:20}
t_{n+1} = \lim_{m\to\infty}\left[\frac{(-1)^{m}}{2}\left(2^{2m}-\frac{1}{(2m-1)!}\log (|\zeta|)^{(2m)}\big(\frac{1}{2}\big)-\frac{1}{2^{2m}}\zeta(2m,\frac{5}{4})\right)-\sum_{k=1}^{n}\frac{1}{t_{k}^{2m}}\right]^{-\frac{1}{2m}}
\end{equation}
as anticipated in (11). Also, an extensive numerical computation of (158) to high precision is summarized in [13], and also for higher order non-trivial zeros.

One limitation for all of these formulas for non-trivial zeros is when $n\to\infty$, then the average gap between zeros gets smaller as $t_{n+1}-t_{n}\sim\frac{2\pi}{\log(n)}$, making the use of these formulas progressively harder and harder to compute the next zero recursively.

By putting these results together, we have two main generalized zeta series, the first series is for the complex magnitude (109) over all zeros (including the hypothetical zeros off of the critical line) as

\begin{equation}\label{eq:20}
Z_{|nt|}(s)= \frac{1}{(\sigma_1^2+t_1^2)^{s}}+\frac{1}{(\sigma_2^2+t_2^2)^{s}}+\frac{1}{(\sigma_3^2+t_3^2)^{s}}+\ldots,
\end{equation}
from which we have an asymptotic relationship
\begin{equation}\label{eq:20}
[Z_{|nt|}(s)]^{-\frac{1}{s}} \sim \sigma_1^2+t_1^2 \quad (s\to\infty).
\end{equation}
The second formula is for generalized zeta series over the imaginary parts

\begin{equation}\label{eq:20}
Z_{1}(2s) = \frac{1}{t_1^{2s}}+\frac{1}{t_2^{2s}}+\frac{1}{t_3^{2s}}+\ldots
\end{equation}
from which we have asymptotic relationship
\begin{equation}\label{eq:20}
[Z_{1}(2s)]^{-\frac{1}{s}} \sim t_1^2.
\end{equation}
Combining (160) and (162) we obtain a true asymptotic formula for the real part of the first non-trivial zero
\begin{equation}\label{eq:20}
\Re(\rho_{1,nt})=\sigma_1=\lim_{s\to\infty}\sqrt{[Z_{|nt|}(s)]^{-\frac{1}{s}}-[Z_{1}(2s)]^{-\frac{1}{s}}}
\end{equation}
and further, by substituting (104) for $Z_{|nt|}(s)$ we further obtain

\begin{equation}\label{eq:20}
\Re(\rho_{1,nt})=\sigma_1=\lim_{s\to\infty}\bigg[\Big(\frac{1}{2}Z_{nt}^2(s)-\frac{1}{2}Z_{nt}(2s)\Big)^{-\frac{1}{s}}-Z_{1}(2s)^{-\frac{1}{s}}\bigg]^{\frac{1}{2}}= \frac{1}{2}.
\end{equation}

Earlier we remarked that the Voros's closed-form formula for $Z_1({2s})$ depends on (RH), and the formula for $Z_{|nt|}(s)$ in terms of the eta constants does not, hence, if the limit converges to $\frac{1}{2}$, it would imply (RH). The convergence is achieved by a cancelation of $t_1$ generated by both equations (160) and (162). In Table 3, we compute $\Re(\rho_{1,nt})$ by equation (164) for various values of the limit variable $m$ from low to high, and observe convergence to $\frac{1}{2}$ as $m$ increases. Of course, a numerical computation in this case cannot be used as a definite  proof of (RH), but interestingly, this formula for the real part is indeed converging to $\frac{1}{2}$, where it is a known zero.

\begin{table}[hbt!]
\caption{The computation of the real part of the first non-trivial zero $\Re(\rho_{1,nt})$ by equation (164) (first 30 decimal places).}
\centering
\begin{tabular}{c c c}
\hline\hline
$m$  & $\Re(\rho_{1,nt})$ & Significant Digits\\ [0.5ex]
\hline
$15$ & 0.\underline{4}73092533136919315298424867840 & 1 \\
$20$ & 0.\underline{4}89872906754757867871088167822 & 1  \\
$25$ & 0.49\underline{9}306593693622997849224832930 & 3  \\
$50$ & 0.5000000\underline{0}2854988386875132586206 & 8  \\
$100$ &0.499999999999999\underline{9}68130042946283 & 16  \\
$150$ &0.500000000000000000000000\underline{0}39540 & 25  \\
$200$ &0.499999999999999999999999999999 & 35  \\
\hline
\end{tabular}
\label{table:nonlin}
\end{table}

\section{The inverse Riemann zeta function}
In the previous section, we outlined the full solution set to

\begin{equation}\label{eq:20}
w=\zeta(s)=0
\end{equation}
(assuming RH), which can also be interpreted as an inverse of

\begin{equation}\label{eq:20}
s=\zeta^{-1}(0)
\end{equation}
as a set of all points $s$ such that $w=\zeta(s)=0$. Now, for other values of $w$-domain of the Riemann zeta function, we seek to find $s$ such that

\begin{equation}\label{eq:20}
s=\zeta^{-1}(w)
\end{equation}
which implies that

\begin{equation}\label{eq:20}
\zeta^{-1}(\zeta(s)) = s
\end{equation}
and
\begin{equation}\label{eq:20}
\zeta(\zeta^{-1}(w))=w
\end{equation}
for some domains $w$ and $s$. Again, the zeta function can have many solutions $s_n$ for which $w=\zeta(s_n)$ (just like for the zeros), and so the inverse zeta is a multi-valued function. Hence, we need to solve an equation
\begin{equation}\label{eq:20}
\zeta(s)-w=0
\end{equation}
as a function of variable $w$. Then, by employing the $m$th log-derivative method and the recursive root extraction described earlier, we can arrive at a solution to (170). To illustrate this, we re-consider the recurrence formula for trivial zeros (85) again as

\begin{equation}\label{eq:20}
\begin{aligned}
\rho_{t,n+1}= \lim_{m\to\infty}\pm \Bigg[-&\frac{1}{(2m-1)!}\frac{d^{(2m)}}{ds^{(2m)}}\log\Big[\zeta(s)(s-1)\Big]\Bigr\rvert_{s\to 0}-\sum_{k=1}^{n}\frac{1}{\rho_{t,k}^{2m}}+\\ &-\sum_{k=1}^{\infty}\Bigg(\frac{1}{\rho_{nt,k}^{2m}}+\frac{1}{\bar{\rho}_{nt,k}^{2m}}\Bigg)\Bigg]^{-\frac{1}{2m}}
\end{aligned}
\end{equation}
(since $Z_{t}$ is dominating the series), and comparing (170) with (171), we solve this equation by replacing trivial zeros with $s_n$ as

\begin{equation}\label{eq:20}
\begin{aligned}
s_{n+1}=\zeta^{-1}(w)= \lim_{m\to\infty}\pm \Bigg[-&\frac{1}{(2m-1)!}\frac{d^{(2m)}}{ds^{(2m)}}\log\Big[(\zeta(s)-w)(s-1)\Big]\Bigr\rvert_{s\to 0}+ \\
&-\sum_{k=1}^{n}\frac{1}{s_k^{2m}}-\sum_{k=1}^{\infty}\Bigg(\frac{1}{\rho_{nt,k}^{2m}}+\frac{1}{\bar{\rho}_{nt,k}^{2m}}\Bigg)\Bigg]^{-\frac{1}{2m}}
\end{aligned}
\end{equation}
where $s_n$ is the multi-valued solution of $s$-domain, as indexed by variable $n$, and which is extracted from the recurrence relation of (172), where $s=s_1$ is the principal solution. But the non-trivial zero terms were only due to case when $w=0$, and so we drop the non-trivial zero terms and obtain the form:
\begin{equation}\label{eq:20}
s_{n+1}=\zeta^{-1}(w)=\lim_{m\to\infty}\pm \left[-\frac{1}{(2m-1)!}\frac{d^{(2m)}}{ds^{(2m)}}\log\left[(\zeta(s)-w)(s-1)\right]\Bigr\rvert_{s\to 0}-\sum_{k=1}^{n}\frac{1}{s_k^{2m}}\right]^{-\frac{1}{2m}}
\end{equation}
and the principal solution is

\begin{equation}\label{eq:20}
s=s_{1}=\zeta^{-1}(w)=\lim_{m\to\infty}\pm \left[-\frac{1}{(m-1)!}\frac{d^{(m)}}{ds^{(m)}}\log\left[(\zeta(s)-w)(s-1)\right]\Bigr\rvert_{s\to 0}\right]^{-\frac{1}{m}},
\end{equation}
where here, we consider an even and odd $m$, and remove the $2m$ for convenience. We next seek to verify this formula by performing a high precision numerical computation of (174) in PARI/GP software package for various test cases. The script that we run is a slight modification of Listing 2, as shown in Listing 5.

\newpage
\lstset{language=C,deletekeywords={for,double,return},caption={PARI script for computing the inverse zeta by equation (174).},label=DescriptiveLabel,captionpos=b}
\begin{lstlisting}[frame=single]
{
   \\ set limit variable
   m = 40;

   \\ set a value for w-domain
   w = zeta(2);

   \\ compute generalized zeta series
   A = -derivnum(s=0,log((zeta(s)-w)*(s-1)), m);
   B = 1/factorial(m-1);
   Z = A*B;

   \\ compute s-domain
   s = Z^(-1/m);
   print(s);
}
\end{lstlisting}

In the first example, we attempt is to invert the Basel problem

\begin{equation}\label{eq:20}
w=\zeta(2)=\frac{\pi^2}{6}=1.64493406684822643647\ldots
\end{equation}
by computing (174) for $m=40$, we obtain

\begin{equation}\label{eq:20}
s=\zeta^{-1}(\frac{\pi^2}{6})=2.000000000000000000000000000\underline{0}534151435532\ldots
\end{equation}
which is accurate to 28 digits after the decimal place. And as $m$ increases, the result clearly converges to $2$. In the next example, we invert the Ap\'ery's  constant

\begin{equation}\label{eq:20}
w=\zeta(3)=1.20205690315959428539\ldots
\end{equation}
then for $m=40$ we compute

\begin{equation}\label{eq:20}
s=\zeta^{-1}(\zeta(3))=3.0000000000000000000\underline{0}22140790061640438069\ldots
\end{equation}
accurate to  20 decimal places, where it is seen converging to $3$ (even for lower values of limit variable $m$, the convergence is fast). In Table 4 we summarize computations for various other values of $w$-domain, where we can see the correct convergence to the inverse Riemann zeta function for $m=20$ every time. For $w=\zeta(0)=-\frac{1}{2}$ there is a singularity at higher derivatives, so we take $\lim_{w\to -\frac{1}{2}}\zeta^{-1}(w)$, and for $\Re(w)\in (-0.5,j_1)\cup \{\Im(w)=0\}$ where $j_1=0.00915989\ldots$ is a constant, there is also a sign change from positive to negative due to this branch, so that the output will come out negative as shown by numerical computations in Table 4. In general, we find that for $\Re(w)\in (-\infty,-0.5)\cup (1,\infty)$ we consider the positive solution

\begin{equation}\label{eq:20}
s=s_{1}=\zeta^{-1}(w)=\lim_{m\to\infty} +\left[-\frac{1}{(m-1)!}\frac{d^{m}}{ds^{m}}\log\left[(\zeta(s)-w)(s-1)\right]\Bigr\rvert_{s\to 0}\right]^{-\frac{1}{m}}
\end{equation}
and otherwise for $\Re(w)\in (-0.5,j_1)\cup \{\Im(w)=0\}$ we consider the negative solution

\begin{equation}\label{eq:20}
s=s_{1}=\zeta^{-1}(w)=\lim_{m\to\infty} -\left[-\frac{1}{(m-1)!}\frac{d^{m}}{ds^{m}}\log\left[(\zeta(s)-w)(s-1)\right]\Bigr\rvert_{s\to 0}\right]^{-\frac{1}{m}}.
\end{equation}
We observe that convergence is faster near $w=-0.5$ for both sides, and as $w\to-0.5$ we get convergence to $0$ as desired. Furthermore, we observe that near both sides of the pole at $s=1$ we can recover the inverse zeta. And hence, when we compute for higher limit variable $m$, the values are clearly converging to the inverse of the Riemann zeta function. In Table 5, we also compute the inverse zeta for various arbitrary values of $w$-domain for $m=100$.

\begin{table}[hbt!]
\caption{The computation of inverse zeta $s=s_1=\zeta^{-1}(w)$ for $m=20$ by equation (174) for different values of $w$. For $w\in (-\infty,-0.5) \cup (1,\infty)$ we consider positive solutions, otherwise for $w\in(-0.5,j_1)$ we consider negative solutions.}
\centering
\begin{tabular}{c c c c}
\hline\hline
$s$ & $w=\zeta(s)$ & $s=\zeta^{-1}(w)$ (First 15 Digits)  & Significant Digits\\ [0.5ex]
\hline 
 $-5$ &   -0.003968253968253 & -1.8847413\underline{7}7602060 & 8 \\
 $-4$ & 0 &                    -1.999999\underline{9}04603844 & 7 \\
 $-3$ &    0.008333333333333 &   -2.4701\underline{6}8918790366 & 5 \\
 $-2$ & 0 &                    -1.999999\underline{9}04603844 & 7 \\
 $-1.5$ & -0.025485201889833 & -1.4999999999\underline{9}8134 & 11 \\
 $-1$ &   -0.083333333333333 &    -1.000000000000000 & 16 \\
 $-0.5$ &    -0.207886224977354 & -0.499999999999999 & 23 \\
 $-0.125$ &  -0.399069668945045 & -0.125000000000000 & 36 \\
 $-0.001$ &  -0.499082063645236 &  0.000999999999999 & 42 \\
 $0.001$ &   -0.500919942713218 &  0.000999999999999 & 42\\
 $0.125$ &   -0.632775623498695 & 0.125000000000000 & 36 \\
 $0.5$   &   -1.460354508809586 &      0.500000000000000 & 26 \\
 $0.75$  &     -3.441285386945222   &  0.749999999999999 & 22 \\
 $0.9999$& -9999.422791616731466 &     0.999900000000000 & 27 \\
 $1.0001$& 10000.577222946437629 &     1.000099999999999 & 26 \\
 $1.5$ & 2.612375348685488 &           1.500000000000000 & 18 \\
 $2$ &   1.644934066848226&            1.9999999999999\underline{9}7 & 14 \\
 $2.5$ & 1.341487257250917&            2.50000000000\underline{0}706 & 12 \\
 $3$ &   1.202056903159594&            3.000000000\underline{0}32817 & 10 \\
 $4$ &   1.082323233711138&            4.0000000\underline{0}8467328 & 8 \\
 $5$ &   1.036927755143369&            5.0000\underline{0}1846688341 & 5 \\
\hline
\end{tabular}
\label{table:nonlin}
\end{table}

\begin{table}[hbt!]
\caption{The computation of inverse zeta $s=s_1=\zeta^{-1}(w)$ for $m=100$ by equation (174) for different values of $w$. For $w\in (-\infty,-0.5) \cup (1,\infty)$ we consider positive solutions, otherwise for $w\in(-0.5,j_1)$ we consider negative solutions. The red color indicates the singularity region where convergence is erroneous.}
\centering
\begin{tabular}{c c c}
\hline\hline
$w$  & $s=\zeta^{-1}(w)$   & $w=\zeta(\zeta^{-1}(w))$ \\ [0.5ex]
\hline 
-10 & 0.90539516131918826348 & -10.00000000000000000000 \\
-5  & 0.82027235216804898973 & -5.000000000000000000000 \\
-4  & 0.78075088259313749868 & -4.000000000000000000000 \\
-3  & 0.71881409407526189655 & -3.000000000000000000000 \\
-2  & 0.60752203756637705289 & -2.000000000000000000000 \\
-1  & 0.34537265729115398953 & -1.000000000000000000000 \\
-0.5001  & 0.00010880828067160644 & -0.50009999999999999999\\
-0.4999  & -0.00010883413591990730 & -0.49989999999999999999\\
-0.1  & -0.90622982899228246768 & -0.100000000000000000000 \\
0   & -1.99999999999999999999 &  0 \\
0.001 & -2.03407870819025354208 & 0.00099999999999999999 \\
0.0015 & -2.05213532171740716650 &0.00149999999999999999 \\
0.0091598 & -2.69835815770380622679 & 0.0091\underline{5}551952718300130 \\
\color{red}{0.01} & 2.69182425874263410494 &  1.27522086147958091320 \\
\color{red}{0.02} & 2.68341537834567817177 &  1.27769681556903809338 \\
\color{red}{0.1}  & 2.62327826166715651687 &  1.29626791092230654966 \\
\color{red}{0.5}  & 3.28523402279617101762 &  1.15403181697782434872 \\
\color{red}{0.8}  & 4.35892653933022255726 &  1.06086646037035161615 \\
\color{red}{0.999} & 9.19090684760189275051 & 1.00175563731403047546 \\
1.001 & 9.19454270908484711549 & 1.00\underline{1}75114882142955996 \\
1.01 &  6.75096988949758004724 & 1.0100000000000000\underline{1}556 \\
1.1 &   3.77062121683766280843 & 1.10000000000000000000 \\
2 &     1.72864723899818361813 & 2.00000000000000000000 \\
3 &     1.41784593578735729296 & 3.00000000000000000000 \\
4 &     1.29396150555724361741 & 4.00000000000000000000 \\
5 &     1.22693680841631476071 & 5.00000000000000000000 \\
10 &    1.10621229947483799036 & 10.0000000000000000000 \\
\hline
\end{tabular}
\label{table:nonlin}
\end{table}

In Figure 3, we plot $s_1=\zeta^{-1}(w)$ for the principal solution by equation (174). The function reproduces the inverse zeta correctly everywhere except in a region $\Re(w)\in(j_1,1)$ where the convergence is erroneous due to (possibly) an infinite number of branch non-isolated singularities that form in an interval $\Re(w)\in(j_1,1)$.  That is not to say that $\zeta(s)$ doesn't have an inverse in this strip, for example, we have $\zeta(-15.48765247\dots)=0.5$ so that $\zeta^{-1}(0.5)=-15.48765247\dots$, but it doesn't exist on the principal branch $s_1$.

\begin{figure}[h]
  \centering
  \includegraphics[width=170mm]{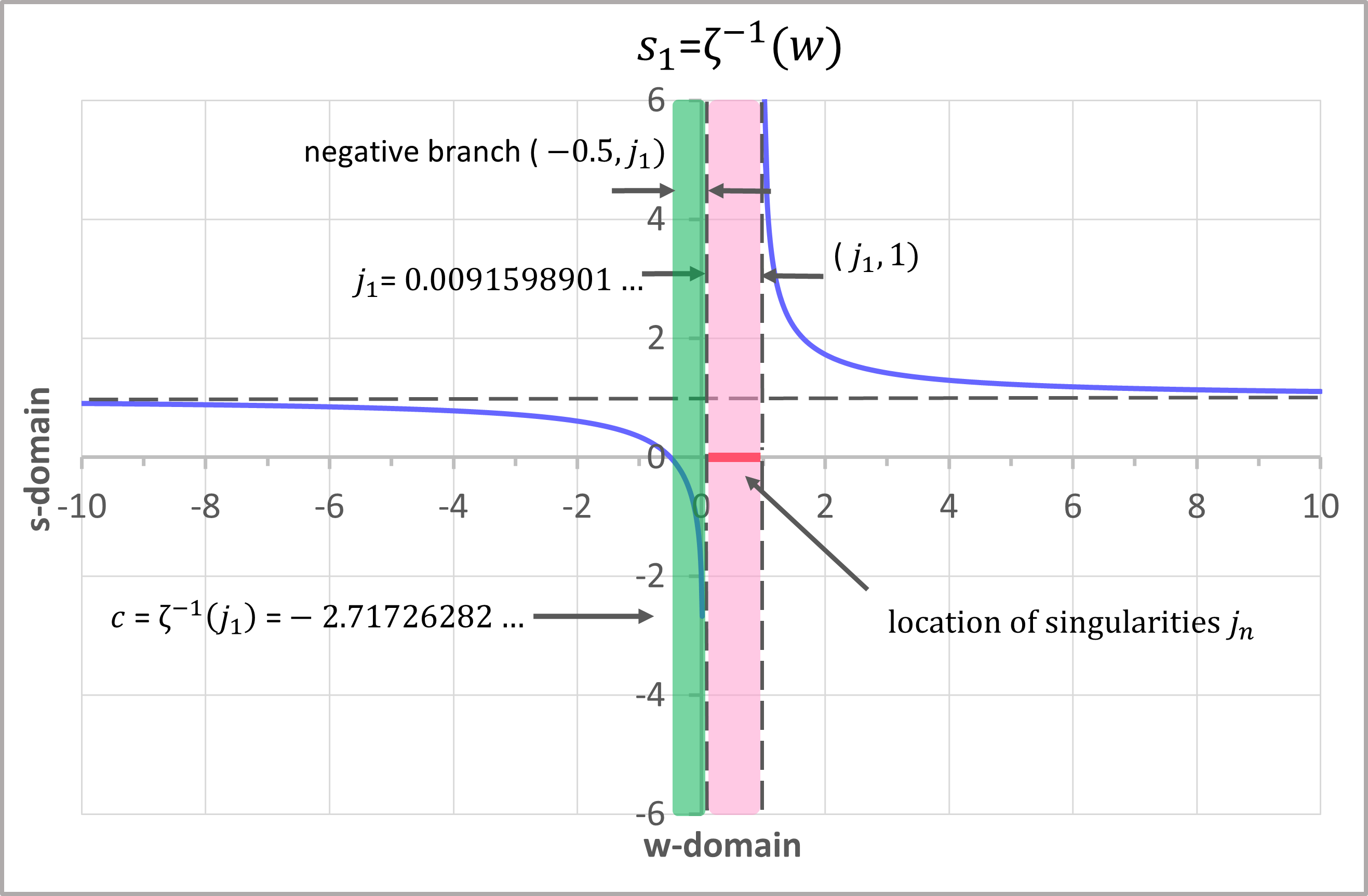}\\
  \caption{A plot of $s_1=\zeta^{-1}(w)$ for $w\in (-10,10)$ by equation (174) showing location of zeros and singularities.}\label{1}
\end{figure}

\begin{figure}[h]
  \centering
  \includegraphics[width=170mm]{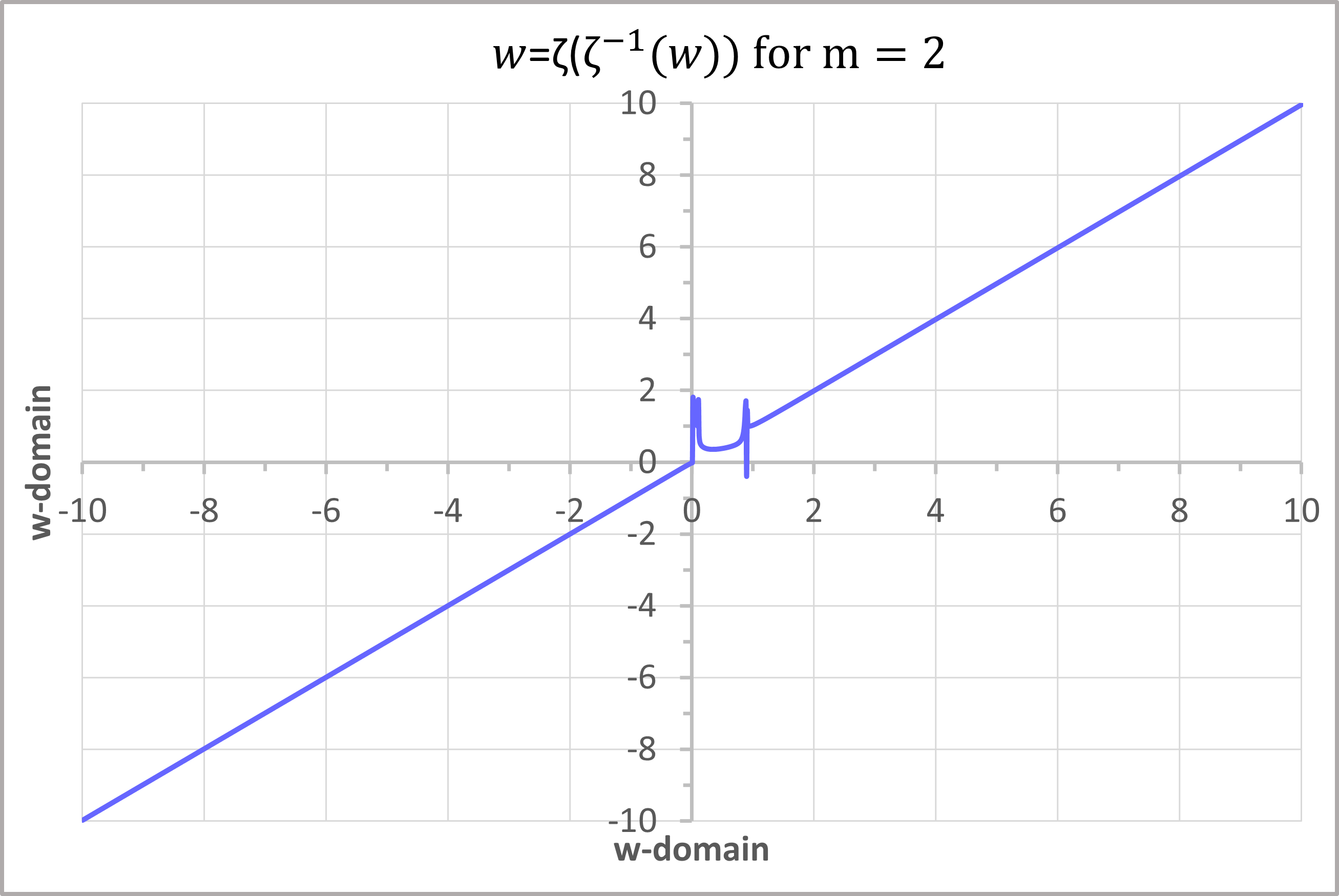}\\
  \caption{A plot of $w=\zeta(\zeta^{-1}(w))$ for real $w\in (-10,10)$ for $m=2$ by the 2nd order approximation equation (185). }\label{1}
\end{figure}

In the next example, we seek to compute the next branch recursively. Let us first compute an inverse of $\zeta(-3)=\frac{1}{120}$ again for $m=40$ and obtain

\begin{equation}\label{eq:20}
s_1=\zeta^{-1}(\frac{1}{120})=-2.4727\underline{0}347315140943243\ldots
\end{equation}
where here, we consider the negative solution. At first, one might wonder that the result is incorrect, but in fact, it is only the principal solution.  The second solution for $\zeta^{-1}(\frac{1}{120})$ is the value that we anticipate, but we recall that the $m$th log-derivative generates the generalized zeta series of over all zeros of a function, hence we can recursively obtain the second solution as

\begin{equation}\label{eq:20}
\begin{aligned}
s_2&=\zeta^{-1}(\frac{1}{120}) = -\lim_{m\to \infty}\Bigg[-\frac{1}{(2m-1)!}&\frac{d^{(2m)}}{dx^{(2m)}}\log\left((\zeta(x)-\frac{1}{120})(x-1)\right)\Bigr\rvert_{x\to 0}+\\
     &           & -\frac{1}{(-2.4727305901\ldots)^{2m}}\Bigg]^{-\frac{1}{2m}} \\
    &=  -3.0000000\underline{0}597327044430\ldots
\end{aligned}
\end{equation}
by removing the first solution, and for $m=20$ the computation converges to a value $-3$ to within 8 decimal places.  But as we mentioned before, such a computation is getting more difficult because it requires the first branch $s_1$ to be known to very high precision in order to ensure convergence. Hence, we pre-computed $s_1$ to $1000$ decimal places using the standard root finder in PARI, since it is more efficient than using (174) for higher $m$. As a result, by knowing $s_1$ accurately, we compute $s_2$ using the recurrence formula. Hence using this process, we can recursively compute all the solutions which lie on different branches, but as we will show later, one has to consider a complex $m$th root in order to access them. But again, numerical computation becomes difficult as very high arbitrary precision is required.

Moving on, if we take $m=2$ and expand the inverse zeta formula (174) as

\begin{equation}\label{eq:20}
\zeta^{-1}(w)\approx\left[\frac{1}{(w+\frac{1}{2})^2}\Bigg(w^2+w\Big(-2\zeta(0)+\zeta^{\prime\prime}(0)\Big)+\zeta(0)^2+\zeta^{\prime}(0)^2-\zeta^{}(0)\zeta^{\prime\prime}(0)\Bigg)\right]^{-\frac{1}{2}}
\end{equation}
and using the identities

\begin{equation}\label{eq:20}
\begin{aligned}
\zeta(0) &= -\frac{1}{2},\\
\zeta^{\prime}(0) &= -\frac{1}{2}\log(2\pi),\\
\zeta^{\prime\prime}(0) &= \frac{1}{2}\gamma^2+\gamma_1-\frac{1}{24}\pi^2-\frac{1}{2}\log^2(2\pi),
\end{aligned}
\end{equation}
we then obtain the 2nd order approximation:
\begin{equation}\label{eq:20}
\zeta^{-1}(w)\approx\pm(w+\frac{1}{2})\Bigg[w^2+w(1+\frac{1}{2}\gamma^2+\gamma_1-\frac{1}{24}\pi^2-\frac{1}{2}\log^2(2\pi))+\frac{1}{4}+\frac{1}{4}\gamma^2+\frac{1}{2}\gamma_1-\frac{\pi^2}{48}\Bigg]^{-\frac{1}{2}}.
\end{equation}

We collect these results into expansion coefficients for $m=2$ as:

\begin{equation}\label{eq:20}
\begin{aligned}
I_0(2) &=\frac{1}{4}+\frac{1}{4}\gamma^2+\frac{1}{2}\gamma_1-\frac{\pi^2}{48}\\
       &= 0.09126979985406300159\ldots, \\
       \\
I_1(2) &= 1+\frac{1}{2}\gamma^2+\gamma_1-\frac{1}{24}\pi^2-\frac{1}{2}\log^2(2\pi)\\
       &=-1.00635645590858485121\ldots,  \\
\\
I_2(2) &=1,
\end{aligned}
\end{equation}
and then re-write (174) more conveniently as

\begin{equation}\label{eq:20}
\begin{aligned}
\zeta^{-1}(w)\approx \pm(w+\frac{1}{2})\Big[I_2(2)w^2+I_1(2)w+I_0(2)\Big]^{-\frac{1}{2}}.
\end{aligned}
\end{equation}
Here, we've added a $\pm$ sign which is dependent on the branch (usually due to the $(w+\frac{1}{2})$ term that must be positive for $w<-0.5$). This second order approximation above is very accurate for a variety of input argument (even complex). For example, for $w=2$ we compute

\begin{equation}\label{eq:20}
\begin{aligned}
\zeta^{-1}(2)\approx 1.7340397592898484279\ldots.
\end{aligned}
\end{equation}
and to verify $\zeta(\zeta^{-1}(2))\approx 1.9902700570\ldots$ is accurate to $2$ significant digits. In Figure 4 we plotted the function $w=\zeta(\zeta^{-1}(w))$ for the 2nd order approximation and see how $w$ is recovered, except in a small region $(j_1,1)$ where we get an erroneous result.  And similarly, for complex argument for $w=2+i$ we compute

\begin{equation}\label{eq:20}
\begin{aligned}
\zeta^{-1}(2+i)\approx 1.4690117151\dots - i0.3470428878\ldots,
\end{aligned}
\end{equation}
and to verify $\zeta(\zeta^{-1}(2+i))\approx 1.9886804524\ldots + i0.9958475706\ldots$ we recover $w$ correctly also to within $2$ significant digits. We will investigate the complex argument in more details a little later. Furthermore, the 2nd degree polynomial in (187) can be factored into its zeros as

\begin{equation}\label{eq:20}
\begin{aligned}
\zeta^{-1}(w)\approx \pm(w+\frac{1}{2})\Big[(w-j_1)(w-j_2)\Big]^{-\frac{1}{2}}
\end{aligned}
\end{equation}
where $j_1= 0.1007872126\ldots$ is the first zero, and $j_2= 0.9055692433\ldots$ is the second zero (computed by solving a quadratic equation). We note that these are the zeros of a polynomial in (187), and hence, they are the branch singularities of 2nd order approximation of $\zeta^{-1}(w)$. To investigate the higher order expansions for $\zeta^{-1}(w)$ in terms of these polynomials $I_n(m)$, can be written with coefficients in terms of Stieltjes constants and incomplete Bell polynomials $\textbf{B}_{n,k}(x_1,x_2,x_3\ldots,x_n)$ due to the Fa\`{a}di-Bruno expansion formula for the $n$th derivative
\begin{equation}\label{eq:20}
\frac{d^n}{dx^n}f(g(x))=\sum_{k=1}^{n}f^{(k)}(g(x))\textbf{B}_{n,k}(g\sp{\prime}(x),g\sp{\prime\prime}(x),\ldots,g^{n-k+1}(x)),
\end{equation}
and if we take
\begin{equation}\label{eq:20}
f(x) = \log(x)
\end{equation}
and
\begin{equation}\label{eq:20}
f^{(n)}(x) = (-1)^{n+1}(n-1)!\frac{1}{x^{n}}.
\end{equation}
Such Bell polynomial expansion will lead to long and complicated expressions for the $m$th log-derivative, so we will not pursue them in this paper. But for the moment, we will just rely on numerical computations, and so, based on (187), we deduce the following asymptotic expansion

\begin{equation}\label{eq:20}
\left[\frac{\zeta^{-1}(w)}{(w+\frac{1}{2})}\right]^{-m}\sim \sum_{n=0}^{m} I_n(m) w^n
\end{equation}
into a $m$th degree polynomial as $m\to \infty$, where $I_n(m)$ are the expansion coefficients. This leads to the series expansion of the inverse zeta function

\begin{equation}\label{eq:20}
\zeta^{-1}(w)=\lim_{m\to\infty}\pm(w+\frac{1}{2})\left(\sum_{n=0}^{m}I_n(m)w^n\right)^{-\frac{1}{m}}
\end{equation}
by these $I_n(m)$ coefficients, which are a function of a limit variable $m$ whose values vary depending on $m$. In Table 6, we compute these coefficients (for several $m$) for further study, and observe the following. For $w=0$, $\zeta^{-1}(w)=-2$ is the first trivial zero, hence we deduce that

\begin{equation}\label{eq:20}
I_0(m)\sim (2\rho_{t,1})^{-m}\sim (-1)^m\frac{1}{2^{2m}}.
\end{equation}
From Table 6 we also observe the asymptotic limits

\begin{equation}\label{eq:20}
I_m(m)\sim 1\quad \text{and}\quad I_{m-1}(m)\sim -\frac{m}{2}.
\end{equation}

\begin{table}[hbt!]
\caption{The computation of expansion coefficients $I_n(m)$ of equation (194) for even $m$}.
\centering
\begin{tabular}{c c c c c}
\hline\hline
$I_n(m)$  & $m=2$ & $m=4$ & $m=6$ & $m=8$ \\ [0.5ex]
\hline
$n=0$ & 0.0912697998 & 0.0042324268 &   0.0002483703 & 0.00001532100 \\
$n=1$ &-1.0063564559 &-0.1967919743 &  -0.0204091776 &-0.00174497183 \\
$n=2$ &  1           & 1.1920976317 &   0.3162826334 & 0.04840981341 \\
$n=3$ &              &-1.9995171980 &  -1.5828262271 &-0.48669059013 \\
$n=4$ &              & 1            &   3.2866782629 & 2.21705837605 \\
$n=5$ &              &              &  -3.0000078068 &-5.15932314768 \\
$n=6$ &              &              &   1            & 6.38227467998 \\
$n=7$ &              &              &                &-4.00000004611 \\
$n=8$ &              &              &                & 1             \\
\hline
\end{tabular}
\label{table:nonlin}
\end{table}

\noindent As $m\to \infty$, this expansion generates an infinite degree polynomial, which also will have infinite zeros $j_n$ which we will next glimpse numerically as a generated attractor of branch singularities. We re-write (194) as factorization

\begin{equation}\label{eq:20}
\left[\frac{\zeta^{-1}(w)}{(w+\frac{1}{2})}\right]^{-m}\sim \prod_{n=1}^{m}\left(w-j_n\right)
\end{equation}
in terms of these zeros, and compute them for $m=4$ in Table 7 and for $m=10$ in Table 8,  using a standard polynomial root finder for a generated polynomial in (194). In Appendix A, we also give values of $j_n$ for $m=50$ in Table 11 as a reference. It is also much better to see $j_n$'s graphically in Figure 5 (where we plot them in a complex plane for $m=4$, $m=10$, $m=30$ and $m=50$), and observe that they form an attractor that clusters near the endpoints. The exact values of these zeros are numerically spread out, and as more zeros are generated as a function of $m$ as $m$ increases, their accuracy also increases, but interestingly, they are mostly real, and cluster roughly in an interval $(0,1)$, but we will narrow it down next, and some zeros are also complex that cluster near $w=1$.

\begin{table}[hbt!]
\caption{The computation of $j_n$ singularities for $m=4$.}
\centering
\begin{tabular}{c c c}
\hline\hline
$n$  & $\Re(j_n)$ & $\Im(j_n)$ \\ [0.5ex]
\hline
1 &	0.02519077171287255364 & 0 \\
2 &	0.22387780988390681825 & 0 \\
3 &	0.75055928996119915729 & 0 \\
4 &	0.99988932644430613063 & 0 \\
\hline
\end{tabular}
\label{table:nonlin}
\end{table}

\begin{table}[hbt!]
\caption{The computation of $j_n$ singularities for $m=10$.}
\centering
\begin{tabular}{c c c}
\hline\hline
$n$  & $\Re(j_n)$ & $\Im(j_n)$ \\ [0.5ex]
\hline
1 &	0.01141939762352641311 & 0 \\
2 &	0.03270893154877055459 & 0 \\
3 &	0.08725746253768978834 & 0 \\
4 &	0.18974173730082442926 & 0 \\
5 &	0.35313390831120714095 & 0 \\
6 & 0.57365189826222332925 & 0 \\
7 &	0.80181268425373759307 & 0 \\
8 &	0.95232274935073811513 & 0 \\
9 &	0.99897561465713752103 & -0.00219195619260189999 \\
10 & 0.9989756146571375210 & 0.002191956192601899994 \\
\hline
\end{tabular}
\label{table:nonlin}
\end{table}

\begin{figure}[hbt!]
  \centering
  \includegraphics[width=170mm]{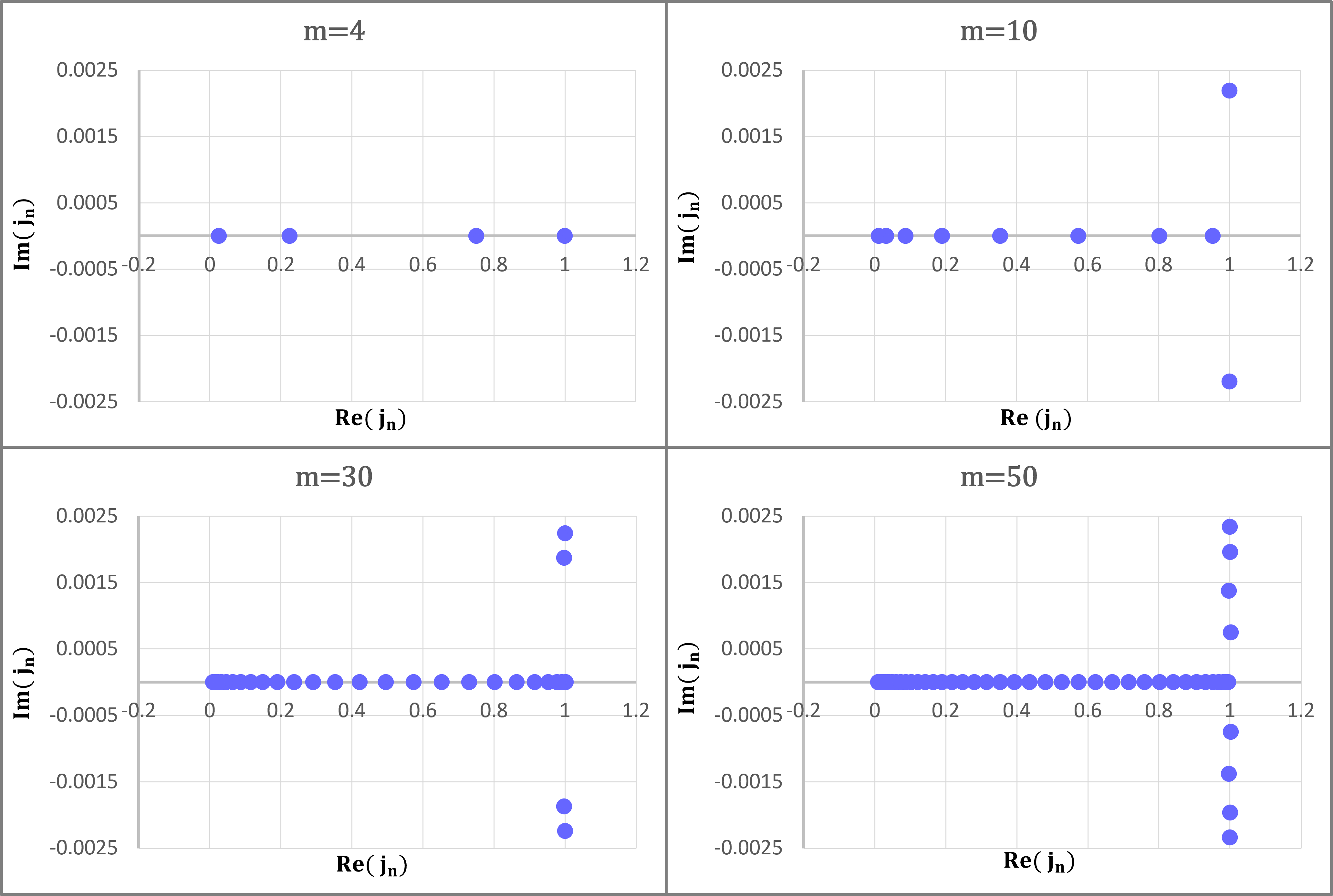}\\
  \caption{An attractor of branch singularities $j_n$ in a complex plane generated for various values of limit variable $m$.}\label{1}
\end{figure}

With more detailed numerical computation, we observe that as $m\to\infty$ they will span the interval $(j_1, j_m)$ where the lower bound is
\begin{equation}\label{eq:20}
j_1(m) \to 0.009159890119903461840056038728\ldots \quad (m\to \infty)
\end{equation}
is the lowest zero, or the principal zero. The value presented was computed numerically to high precision. And the upper bound is
\begin{equation}\label{eq:20}
j_m(m) \to O(1) \quad (m\to \infty)
\end{equation}
due to the pole of $\zeta(1)$.

\begin{figure}[h]
  \centering
  \includegraphics[width=170mm]{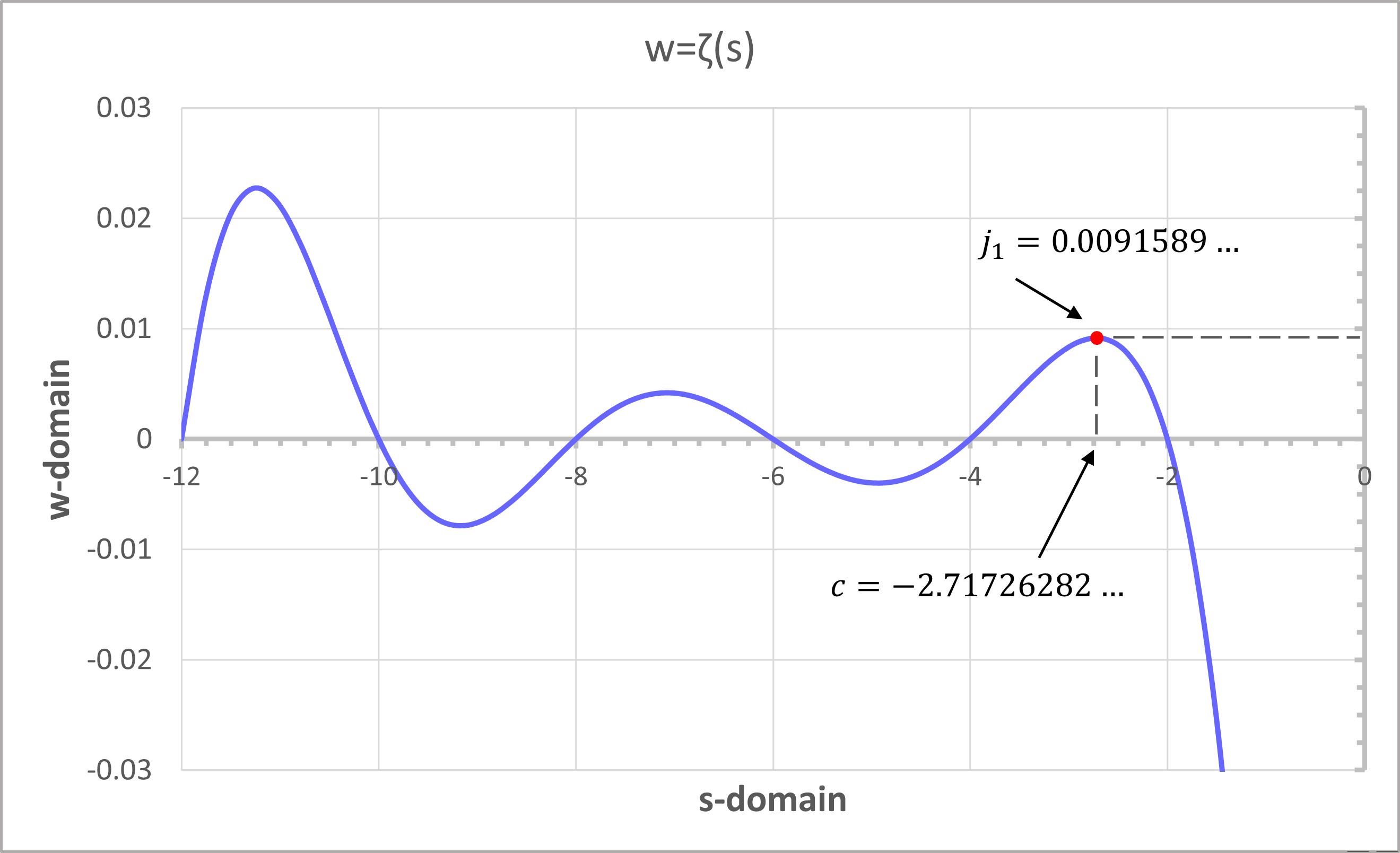}\\
  \caption{A plot of $w=\zeta(s)$ for $s\in (-8,0)$ locating local maxima $j_1$. }\label{1}
\end{figure}

From Figure 6, we see that $j_1$ corresponds to the first local maxima (between $s=-4$ and $s=-2$) in a region where the zeta takes a first turn from being monotonically increasing when going from left to right in $s$-domain ($s=1$ to $s=-2.7172628292\ldots)$ and in $w$-domain as ($w=-\infty$ to $w=0.0091598901\ldots$), at which point causes a discontinuity for this branch. We observe that $j_n$ are zeros of the expansion (194), and this implies at first that they are the singularities of $\zeta^{-1}(w)$, but because they are under an $m$th root that induces $m$ branches and forms an algebraic branch point [14, p.143]. Hence, the strip $(j_1,1)$ fills the remaining $w$-domain gap from $j_1$ to $\zeta(1)$ with these branch singularities. We conjecture that
\newtheorem{conj}{Conjecture}
\begin{conj}\notag
The principal branch $s_1=\zeta^{-1}(w)$ has an infinite number of real branch singularities in a strip $(j_1, 1)$.
\end{conj}
\noindent The inverse zeta can be represented by factorization by these singularities. In Figure 3, we highlighted this singularity strip region in relation to $s_1=\zeta^{-1}(w)$. We will refer to the constants $j_n$ interchangeably as either zeros of (198) or singularities of $s_1=\zeta^{-1}(w)$.

Now, since $j_1$ is the principal zero of the expansion (194), we can find its formula by solving the infinite degree polynomial equation using Theorem 1 and Theorem 2, and find that

\begin{equation}\label{eq:20}
j_1 = \lim_{m\to\infty}\left[\frac{m}{(m-1)!}\frac{d^m}{dw^m}\log\left[\frac{\zeta^{-1}(w)}{(w+\frac{1}{2})}\right]\Bigr\rvert_{w\to 0}\right]^{-\frac{1}{m}}
\end{equation}
and in Table 9, compute (201) for a few values of $m$ and observe a slow convergence to $j_1$. We also consider writing next higher branch singularities by the recurrence relation
\begin{equation}\label{eq:20}
j_{n+1} = \lim_{m\to\infty}\left[\frac{2m}{(2m-1)!}\frac{d^{(2m)}}{dw^{(2m)}}\log\left[\frac{\zeta^{-1}(w)}{(w+\frac{1}{2})}\right]\Bigr\rvert_{w\to 0}-\sum_{k=1}^{n}\frac{1}{j_k^{2m}}\right]^{-\frac{1}{2m}}
\end{equation}
where we consider a $2m$ limit value to avoid an alternating sign in the recurrence, but numerically is very hard to compute since these non-isolated singularities are so densely spaced in an interval $(j_1, 1)$ is almost impossible to extract them. Also, the value for a constant $c$ for which $j_1=\zeta(c)$ is $c=-2.717262829\ldots$, which is close to $e$ to within 3 decimal places, and from this obtain a simple approximation to $j_1$ as

\begin{equation}\label{eq:20}
j_1 \approx \zeta(-e)= 0.009159\underline{8}77559420231\ldots
\end{equation}
which is accurate to within 7 decimal places.

\begin{table}[hbt!]
\caption{The computation of $j_1$ by equation (201) for various $m$ from low to high.}
\centering
\begin{tabular}{c c c}
\hline\hline
$m$  & $j_1$ & Significant Digits\\ [0.5ex]
\hline
$10$ &  0.\underline{0}1141936690297939790 & 1\\
$50$ &  0.00\underline{9}24371071593150307 & 3\\
$100$ & 0.009\underline{1}6896287172313725 & 4\\
\hline
\end{tabular}
\label{table:nonlin}
\end{table}

These relations allow us to write the inverse Riemann zeta function as factorization into zeros and singularities as

\begin{equation}\label{eq:20}
s_1=\zeta^{-1}(w)=\lim_{m\to \infty}\pm(w+\frac{1}{2})\prod_{n=1}^{m}\left(w-j_n\right)^{-\frac{1}{m}}.
\end{equation}

These generated singularities are so finely balanced that even for $m=10$, they can reproduce the inverse zeta function to a great degree of accuracy as we will see shortly. Also, we have the identity
\begin{equation}\label{eq:20}
\prod_{n=1}^{m}\left(j_n\right)^{-\frac{1}{m}}= 4 \quad (m\to \infty)
\end{equation}
that we just infer from numerical computations.
We remarked earlier that some of the singularities in the attractor are also complex and cluster near $1$, as shown in Figure 5 for higher $m$. Initially, we're unsure as to whether these complex zeros are real or artifacts of the root finder, but we find that they play a central role (in conjunction with the real roots) in computing the product formula (204) and many identities that follow. For example, we have
\begin{equation}\label{eq:20}
\sum_{n=1}^{m}j_n\sim \frac{m}{2}\quad (m\to\infty)
\end{equation}
obtained based on expanding the coefficients in (204). From this we have the mean value of $j_n$:
\begin{equation}\label{eq:20}
\lim_{m\to \infty} \frac{1}{m}\sum_{n=1}^{m}j_n = \frac{1}{2},
\end{equation}
and also from (205) another identity
\begin{equation}\label{eq:20}
\lim_{m\to \infty} \frac{1}{m}\sum_{n=1}^{m}\log(j_n)= -2\log(2).
\end{equation}
We observe that as $m$ increases, then the number of complex singularities that are generated increases, but their absolute values tends $1$. This tendency is also captured by Conjecture 1 above. If true, then it would imply that as $m\to\infty$ then these complex singularities will disappear and will only remain on the real line $(j_1,1)$.

We next investigate how the inverse zeta function converges for complex argument.  As another example, we compute the inverse zeta of
\begin{equation}\label{eq:20}
s=\zeta^{-1}(2+i)= 1.466595797094670\ldots-i0.343719739467598\ldots
\end{equation}
for $m=10$, and then, when taking the zeta of the inverse zeta
\begin{equation}\label{eq:20}
w=\zeta(\zeta^{-1}(2+i))=2.000000007384116\ldots+i0.999999997993535\ldots
\end{equation}
we recover the $w$-domain correctly (we see is better approximation than the 2nd order equation (185)). As another example we take the inverse zeta for large input argument
\begin{equation}\label{eq:20}
\begin{aligned}
s=\zeta^{-1}(1234&56789-i987654321)= \\
      &1.000000000124615\ldots+i0.000000000996923\ldots
\end{aligned}
\end{equation}
and then, when taking the zeta of the inverse zeta above, then we compute
\begin{equation}\label{eq:20}
w=\zeta(s)=123456789.01848\ldots-i987654321.14785\ldots
\end{equation}
where we see correct convergence to within $1$ decimal place, but if we re-compute for $m=20$, then we get
\begin{equation}\label{eq:20}
w=\zeta(s)=123456789.00000\ldots-i987654321.00000\ldots,
\end{equation}
which is now accurate to $15$ digits after the decimal place.
In general, we find that for large complex input argument, the convergence is very good, but that is sometimes not the case for smaller input argument, where in many cases, we don't get correct convergence at first. For example, if we evaluate

\begin{equation}\label{eq:20}
s=\zeta^{-1}(1.5+i)= 1.521134764270121\ldots+i0.417327503093697\ldots
\end{equation}
for $m=10$, and then inverting back
\begin{equation}\label{eq:20}
w=\zeta(\zeta^{-1}(1.5+i))=1.783854226864277\ldots-i0.908052465458989\ldots
\end{equation}
we get an erroneous results. The reason is because of the $m$th root involved in the computation of the inverse zeta actually generates $m$ branches. In general, the $m$th root of a complex number $z$ can be written as

\begin{equation}\label{eq:20}
z^{\frac{1}{m}} = |z|^{\frac{1}{m}}e^{i\frac{1}{m}(\arg(z)+2\pi \lambda)}
\end{equation}
where the branch ranges from $\lambda=0\ldots m-1$. So far, we've been using the principal root for $\lambda=0$, which is the standard $m$th root, but for complex numbers, we have to select $\lambda$ for which the solution that we want lies.  We do not have an exact criterion for which $\lambda$ solution to use, so we have to individually check every solution and find the one that we need. For example, in re-computing (214), we find that the $m$th root for $\lambda=9$ gives

\begin{equation}\label{eq:20}
s=\zeta^{-1}(1.5+i)= 1.475922826723574\ldots-i0.556475538964500\ldots
\end{equation}
for $m=10$, and then
\begin{equation}\label{eq:20}
w=\zeta(\zeta^{-1}(1.5+i))=1.500000011509227\ldots+i0.999999987375822\ldots
\end{equation}
finally reproduces the correct result. These results lead us to introducing an error function
\begin{equation}\label{eq:20}
E(w)=|w-\zeta(\zeta^{-1}(w))|
\end{equation}
used to quantify how well the inverse zeta is inverting. Essentially, taking $\zeta(\zeta^{-1}(w))$ should reproduce $w$, and so when subtracting $w$ off, we should expect
\begin{equation}\label{eq:20}
E(w)=0,
\end{equation}
and when computing it numerically, $E(w)$ will be very small because the convergence of $\zeta^{-1}(w)$ is generally very good. But when $\zeta^{-1}(w)$ is not converging correctly, usually due to the $m$th root lying on another branch, then $E(w)$ will be very high in relation to a case when $\zeta^{-1}(w)$ is normally converging. This contrast between high convergence rate and no convergence at all, allows us to write a simple search algorithm to sweep the branch of the $m$th root and minimize $E(w)$. In doing so, we introduced a reasonable threshold value of $t_x=10^{-3}$ to minimize $E(w)$ (which may be re-adjusted), and once the minima has been found, the code exits out of the loop and returns the correct branch. From further numerical study, we found that there is only one branch of the $m$th root giving the correct answer, and all other branches give erroneous results, thus making the use of this loop very easy. In our code, we define a custom $m$th root function in Listing 6, and in Listing 7, we modify the inverse zeta function with the new $m$th root branch search loop. The second modification to the script we made is that now we load a pre-computed table of $j_n$'s from a text file, and evaluate the product formula (204), instead of computing the $m$th derivative using the $\textbf{derivnum}$ function (which is slow for high $m$). In Appendix A, we provide a Table 11 with pre-computed $j_n$ for $m=50$ for reference. Hence, together with the $m$th root function, the presented algorithm allows for a very fast evaluation of $s_1=\zeta^{-1}(w)$ for any complex argument $w$ (in just under several milli-seconds) on a standard workstation. The only requirement is to pre-compute a table of $j_n$ singularities and store them in a file. In contrast, the $\textbf{derivnum}$ function takes 60 ms to evaluate one inversion for $m=10$ on our workstation, and over 5-20 minutes for $m=400$.

\lstset{language=C,deletekeywords={for,double,return},caption={A custom function to compute an $m$th root for an $\lambda$ branch.},label=DescriptiveLabel,captionpos=b}
\begin{lstlisting}[frame=single][hbt!]
\\ define mth root function
\\ s is input argument, m is mth root, lambda is the nth branch
xroot(s,m,lambda)=
{
  r = abs(s);
  y = r^(1/m)*exp(I*arg(s)/m+I*lambda*2*Pi/m);
  return(y);
}
\end{lstlisting}

\noindent When running the new script in Listing 7, we can reproduce all the results in this paper, including for the negative branch for the range $(-0.5,j_1)$ we saw earlier, which actually corresponds to an $m$th root branch at $\lambda=\frac{m}{2}$ if $m$ is even, and which is automatically found by the code. One more example, we invert

\begin{equation}\label{eq:20}
s=\zeta^{-1}(0.5+i)= 0.933314322626762\ldots-i0.930958378790106\ldots
\end{equation}
for $m=10$ which lies just above the singularities $(j_1,1)$ using the new script in Listing 7, then we verify this
\begin{equation}\label{eq:20}
w=\zeta(\zeta^{-1}(0.5+i))=0.500000004914683\ldots+i1.000000012981412\ldots
\end{equation}
which inverts $s$ back correctly, which corresponds to the $m$th root branch of $\lambda=8$ which is automatically found by the code.
\noindent To check more complex points, we generated a density plot of the error function $E(w)$ from equation (219) in Figure 7 by computing it for a grid of complex points $101\times 101$ which contains $10201$ total points spanning a range $\Re(w)\in(-2,2)$ and $\Im(w)\in(-2,2)$ equally spaced for $m=10$, and using the new code in Listing 7 which took a 1-2 minutes to compute all points. We see that generally $E(w)\sim 10^{-8}$ throughout, and when it's close to the zero at $w=-\frac{1}{2}$, then $E(w)\sim 10^{-28}$ which is surprisingly very good, and then when it's near the singularities in the range $(j_1,1)$, then $E(w)$ gets worse (as expected), and then it completely fails at the singularities (blue color). The function still runs in the singularity region because numerically, it's very unlikely to hit an exact location of the singularity, causing a $\frac{1}{0}$ division. In Figure 8, we re-plot again but for $m=50$, and now see much better convergence over the previous case for $m=10$, where now we get $E(w)\sim 10^{-55}$ throughout, and $E(w)\sim 10^{-128}$ close to zero, and then when it's near the singularity region $E(w)\sim 10^{-7}$.

\newpage

\lstset{language=C,deletekeywords={double},caption={A new  PARI function for $\zeta^{-1}(w)$ using the $m$th root branch search and singularity expansion representation (204).},label=DescriptiveLabel,captionpos=b}
\begin{lstlisting}[frame=single][hbt!]
\\ inverse zeta function valid for complex w argument
izeta(w)=
{
   \\ set mth root branch threshold
   tx = 1e-3;

   \\ load singularities from txt file into a vector
   jx = readvec("jx_singularities_m50.txt");

   \\ compute the length of vector
   m = length(jx);

   \\ compute product due to singularities
   A = prod(i=1,m,(w-jx[i]))^(-1);

   \\ mth root branch search
   for(i=0,m-1,

       \\ compute s-domain
       s = (w+1/2)*xroot(A,m,i);

       \\ compute error function
       E = abs(zeta(s)-w);

       \\ exit out of loop when threshold is met
       if(E<tx, break);
   );
   return(s);
}
\end{lstlisting}

\newpage
\begin{figure}[hbt!]
  \centering
  \includegraphics[width=120mm]{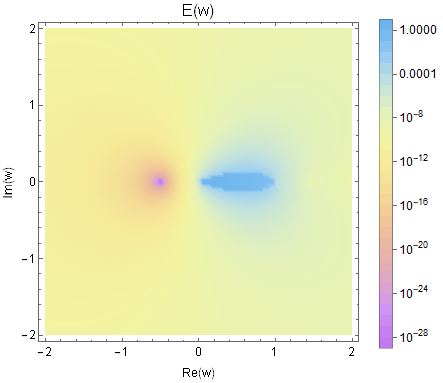}\\
  \caption{A density plot of $E(w)$ by equation (219) for $m=10$ in range of $\Re(w)\in (-2,2)$ and $\Im(w)\in (-2,2)$.}\label{1}
\end{figure}

\begin{figure}[hbt!]
  \centering
  \includegraphics[width=120mm]{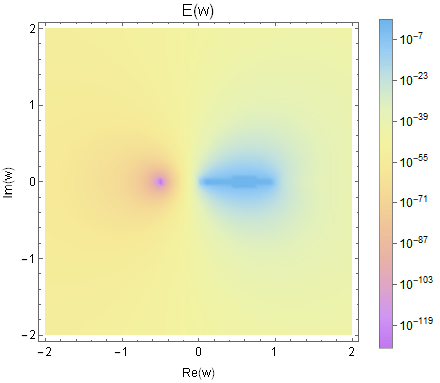}\\
  \caption{A density plot of $E(w)$ by equation (219) for $m=50$ in range of $\Re(w)\in (-2,2)$ and $\Im(w)\in (-2,2)$.}\label{1}
\end{figure}

\newpage

\section{The $\zeta^{-1}(w)$ near its zero}
The first few terms of Taylor expansion coefficients of $\zeta(s)$ about $s=0$ are

\begin{equation}\label{eq:20}
\zeta(s)= -\frac{1}{2}+\zeta'(0)s+\frac{1}{2}\zeta''(0)s^2+\ldots
\end{equation}
and the series converge for $|s|<1$. The value for $\zeta'(0)=-\log\sqrt{2\pi}$. If we write
\begin{equation}\label{eq:20}
\zeta(s)= -\frac{1}{2}-o(s\log\sqrt{2\pi})
\end{equation}
as $s\to0$ near the origin and then taking the inverse zeta of both sides, we deduce that
\begin{equation}\label{eq:20}
s=\zeta^{-1}(-\frac{1}{2}-s\log\sqrt{2\pi})
\end{equation}
as $s\to0$, and now, applying the inverse zeta series (195) given by 

\begin{equation}\label{eq:20}
\zeta^{-1}(w)=\lim_{m\to\infty}-(w+\frac{1}{2})\left(\sum_{n=0}^{m}I_n(m)w^n\right)^{-\frac{1}{m}},
\end{equation}
(the negative branch) to both sides of (226) above yields
\begin{equation}\label{eq:20}
s=\lim_{m\to\infty}-(-\frac{1}{2}-s\log\sqrt{2\pi}+\frac{1}{2})\left(\sum_{n=0}^{m}I_n(m)\left(-\frac{1}{2}-s\log\sqrt{2\pi}\right)^n\right)^{-\frac{1}{m}}
\end{equation}
and the $-\frac{1}{2}$ will cancel. The $s$ variable also cancels on both sides, and we get
\begin{equation}\label{eq:20}
\lim_{m\to\infty}-(-\log(\sqrt{2\pi}))\left(\sum_{n=0}^{m}I_n(m)(-\frac{1}{2}-s\log\sqrt{2\pi})^n\right)^{-\frac{1}{m}}=1
\end{equation}
We find that the remaining $s$ inside the infinite sum becomes negligible, and we obtain an identity

\begin{equation}\label{eq:20}
\lim_{m\to \infty}\left(\sum_{n=0}^{m}I_n(m)(-\frac{1}{2})^n\right)^{\frac{1}{m}}=\log\sqrt{2\pi}.
\end{equation}

\newpage
\section{The asymptotic relations of $\zeta^{-1}(w)$}
We first investigate a limit formula for the Euler-Mascheroni constant. From the Laurent expansion of $\zeta(s)$ in (124), we can deduce a limit identity

\begin{equation}\label{eq:20}
\gamma= \lim_{s\to 1+}\left[\zeta(s)-\frac{1}{s-1}\right]
\end{equation}
and further, by transforming the limit variable $s\to 1+\frac{1}{s}$ we obtain

\begin{equation}\label{eq:20}
\gamma= \lim_{s\to\infty}\zeta(1+\frac{1}{s})-s.
\end{equation}
We empirically find a similar relation for the inverse Riemann zeta function by numerically evaluating for $s=10^4$ as

\begin{equation}\label{eq:20}
\zeta^{-1}(10^4) =  1.000100005772562674143\ldots,
\end{equation}
where we observe a sign of a tailing $\gamma$ in the digits, which is on the order of $O(s^{-2})$. So we deduce that

\begin{equation}\label{eq:20}
\zeta^{-1}(s) \sim  1+\frac{1}{s}+ O\big(\gamma\frac{1}{s^2}\big)\quad (s\to \infty)
\end{equation}
from which we have
\begin{equation}\label{eq:20}
\gamma= \lim_{s\to\infty}\left[\zeta^{-1}(s)-(1+\frac{1}{s})\right]s^2.
\end{equation}
And similarly, we find that

\begin{equation}\label{eq:20}
\gamma= \lim_{s\to\infty}\left[\zeta^{-1}(-s)-(1-\frac{1}{s})\right]s^2
\end{equation}
from which we conclude that
\begin{equation}\label{eq:20}
\zeta^{-1}(s)\sim\zeta^{-1}(-s)\to O(1)
\end{equation}
as it is seen in graph in Figure 3. In Table 10, we summarize numerical computation of (234) using the inverse zeta formula for $m=100$ and observe convergence to $\gamma$.

\begin{table}[hbt!]
\caption{The computation of $\gamma$ by inverse zeta for for various $s$ from low to high by equation (234) and $\zeta^{-1}(s)$ for $m=100$.}
\centering
\begin{tabular}{c c c}
\hline\hline
$s$  & $\gamma$ & Significant Digits\\ [0.5ex]
\hline
$10^1$ & 0.62122994748379903608 & 0 \\
$10^2$ & 0.\underline{5}8130721658646456077 & 1  \\
$10^3$ & 0.57\underline{7}62197248836203702 & 3  \\
$10^4$ & 0.577\underline{2}5626741433442042 & 4  \\
$10^5$ & 0.5772\underline{1}972487058219773 & 5  \\
$10^6$ & 0.5772\underline{1}607089561571393 & 5  \\
$10^7$ & 0.57721\underline{5}70550091292536 & 6  \\
$10^8$ & 0.5772156\underline{6}896147058487 & 8  \\
$10^9$ & 0.5772156\underline{6}530752663021 & 8  \\
$10^{10}$ & 0.577215664\underline{9}4213223753 & 10  \\
$10^{11}$ & 0.5772156649\underline{0}559279829 & 11  \\
$10^{12}$ & 0.57721566490\underline{1}93885437 & 12  \\
\hline
\end{tabular}
\label{table:nonlin}
\end{table}

We can also obtain a different representation by expanding (234) as
\begin{equation}\label{eq:20}
\gamma= \lim_{s\to\infty}\left[s^2\zeta^{-1}(s)-(s^2+s)\right]
\end{equation}
from which we recognize the sum of natural numbers
\begin{equation}\label{eq:20}
\sum_{n=1}^{k}n=1+2+3+\ldots k = \frac{k^2}{2}+\frac{k}{2}
\end{equation}
where we obtain

\begin{equation}\label{eq:20}
\sum_{n=1}^{k}n=1+2+3+\ldots \sim -\frac{1}{2}\gamma+\frac{1}{2}k^2\zeta^{-1}(k)\quad (k\to \infty)
\end{equation}
and this is in contrast to the Euler's relation for harmonic sum
\begin{equation}\label{eq:20}
\sum_{n=1}^{k}\frac{1}{n}=1+\frac{1}{2}+\frac{1}{3}+\ldots \sim \gamma+\log(k) \quad (k\to \infty)
\end{equation}
where here, the term $\frac{1}{2}k^2\zeta^{-1}(k)$ is sort of dual to $\log(k)$ in the sense of a reflection about the origin $\zeta(s)\leftrightarrow\zeta(-s)$ for series (1), when $s\to 1$.

On a side note, it is often loosely written that
\begin{equation}\label{eq:20}
\zeta(-1)= \sum_{n=1}^{\infty}n=1+2+3+\ldots =-\frac{1}{12}
\end{equation}
in the context of the Riemann zeta function and zeta regularization, where the asymptotic term is omitted. We briefly investigate the asymptotic term of (241) by the Euler-Maclaurin formula, which breaks up the series (1) into a partial sum up to the $k-1$ order, and the remainder starting at $k$ and going to infinity
\begin{equation}\label{eq:20}
\zeta(s) = \sum_{n=1}^{k-1}\frac{1}{n^s}+\sum_{n=k}^{\infty}\frac{1}{n^s}
\end{equation}
as shown in [7, p.114], when the Euler-Maclaurin summation formula is applied to the remainder term we get

\begin{equation}\label{eq:20}
\zeta(s) = \sum_{n=1}^{k-1}\frac{1}{n^s}-\frac{k^{1-s}}{1-s}+\frac{1}{2}k^{-s}+\frac{B_2}{2}sk^{-s-1}+O(k^{-s-3})
\end{equation}
and then substituting $s=2$ we get

\begin{equation}\label{eq:20}
\zeta(2) = \sum_{n=1}^{k-1}\frac{1}{n^2}+\frac{1}{k}+\frac{1}{2k^2}+B_2\frac{1}{k^3}+O(\frac{1}{k^5})
\end{equation}
now, when solving for $B_2$ we get

\begin{equation}\label{eq:20}
B_2 = k^3\left(\zeta(2)-\sum_{n=1}^{k-1}\frac{1}{n^2}\right)-k^2-\frac{k}{2}-O(\frac{1}{k^2}),
\end{equation}
and see that is slowly resembling (241), and multiplying by $-\frac{1}{2}$ yields

\begin{equation}\label{eq:20}
-\frac{1}{2}B_2 = \frac{1}{2}k^2+\frac{1}{4}k-\frac{1}{2}k^3\left(\zeta(2)-\sum_{n=1}^{k-1}\frac{1}{n^2}\right)+O(\frac{1}{2k^2}).
\end{equation}
From this, we have the full asymptotic relation
\begin{equation}\label{eq:20}
\sum_{n=1}^{k}n=1+2+3+\ldots -\left[\frac{1}{4}k+\frac{1}{2}k^3\left(\frac{\pi^2}{6}-\sum_{n=1}^{k-1}\frac{1}{n^2}\right)\right]=-\frac{1}{12}
\end{equation}
as $k\to\infty$ which is the correct version of (241) in the context of the Riemann zeta function (involving its analytical continuation) by including the asymptotic term. Collecting these results, we have two asymptotic representations for the sum of natural numbers in the context of the Riemann zeta function as

\begin{equation}\label{eq:20}
\sum_{n=1}^{k}n=1+2+3+\ldots -\frac{1}{2}k^2\zeta^{-1}(k)= -\frac{1}{2}\gamma\quad (k\to \infty)
\end{equation}
and
\begin{equation}\label{eq:20}
\sum_{n=1}^{k}n=1+2+3+\ldots -\left[\frac{1}{4}k+\frac{1}{2}k^3\left(\frac{\pi^2}{6}-\sum_{n=1}^{k-1}\frac{1}{n^2}\right)\right]=-\frac{1}{12} \quad (k\to \infty).
\end{equation}
If we drop the asymptotic terms, and summing to infinity, then we casually write

\begin{equation}\label{eq:20}
\sum_{n=1}^{\infty}n=1+2+3+\ldots = -\frac{1}{2}\gamma,
\end{equation}
and
\begin{equation}\label{eq:20}
\sum_{n=1}^{\infty}n=1+2+3+\ldots =-\frac{1}{12},
\end{equation}
which are only loosely taken at face value and which is implied in the context Riemann zeta function. The complete asymptotic representations are (248) and (249). Also, comparing the values

\begin{equation}\label{eq:20}
\begin{aligned}
-\frac{1}{2}\gamma &= -0.2886078324\ldots, \\
-\frac{1}{12} &=-0.0833333333\ldots.
\end{aligned}
\end{equation}
It is often written in the literature that $-\frac{1}{12}$ is the assigned value to the sum of natural numbers and the asymptotic part is discarded. In actuality, one could arbitrarily assign any value to any divergent series by subtracting two such series with the same growth rate, and where the difference results in a finite constant and the divergent parts cancel.

\section{On the derivatives of $\zeta^{-1}(w)$}
We now consider the derivatives of $s_1=\zeta^{-1}(w)$ as such. By differentiating the inverse function relation

\begin{equation}\label{eq:20}
\zeta(\zeta^{-1}(w))=w
\end{equation}
we get
\begin{equation}\label{eq:20}
\zeta^{\prime}[\zeta^{-1}(w)][\zeta^{-1}(w)]^{\prime}=1
\end{equation}
by the composition rule. And this leads to a simple formula

\begin{equation}\label{eq:20}
[\zeta^{-1}(w)]^{\prime}=\frac{1}{\zeta^{\prime}[\zeta^{-1}(w)]},
\end{equation}
provided that $\zeta^{\prime}(s)\neq 0$. We saw earlier that the constant $s=c=-2.71726282\ldots$ is a zero of $\zeta^{\prime}(c) = 0$, and for which $j_1=\zeta(c)=0.00915989\ldots$. Then, evaluating (255) for $w=0$ we get

\begin{equation}\label{eq:20}
[\zeta^{-1}(0)]^{\prime}=\frac{1}{\zeta^{\prime}[\zeta^{-1}(0)]}=\frac{1}{\zeta^{\prime}[-2]}
\end{equation}
(taking the principal zero), and using the well-known identity

\begin{equation}\label{eq:20}
\zeta^{\prime}(-2n)=\frac{(-1)^n \zeta(2n+1)(2n)!}{2^{2n+1}\pi^{2n}},
\end{equation}
we have
\begin{equation}\label{eq:20}
[\zeta^{-1}(0)]^{\prime}=-\frac{4\pi^2}{\zeta(3)}.
\end{equation}
Recalling the inverse zeta factorization formula (204) as

\begin{equation}\label{eq:20}
\zeta^{-1}(w)=\lim_{m\to \infty}-(w+\frac{1}{2})\prod_{n=1}^{m}\left(w-j_n\right)^{-\frac{1}{m}}
\end{equation}
(the negative branch), and taking the first log-derivative gives

\begin{equation}\label{eq:20}
\frac{[\zeta^{-1}(w)]^{\prime}}{\zeta^{-1}(w)}=\frac{1}{(w+\frac{1}{2})}-\lim_{m\to \infty}\frac{1}{m}\sum_{n=1}^{m}\frac{1}{w-j_n},
\end{equation}
from which we have an alternate formula in terms of branch singularities $j_n$ as

\begin{equation}\label{eq:20}
[\zeta^{-1}(w)]^{\prime}=\zeta^{-1}(w)\left[\frac{1}{(w+\frac{1}{2})}-\lim_{m\to \infty}\frac{1}{m}\sum_{n=1}^{m}\frac{1}{w-j_n}\right].
\end{equation}
If we let $w=0$ then we have

\begin{equation}\label{eq:20}
\frac{[\zeta^{-1}(0)]^{\prime}}{\zeta^{-1}(0)}=2+\lim_{m\to \infty}\frac{1}{m}\sum_{n=1}^{m}\frac{1}{j_n}.
\end{equation}
Now relating with (258) we obtain a formula for either the average value of

\begin{equation}\label{eq:20}
\lim_{m\to \infty}\frac{1}{m}\sum_{n=1}^{m}\frac{1}{j_n}=2\left(\frac{\pi^2}{\zeta(3)}-1\right),
\end{equation}
or a formula for Ap\'ery's constant

\begin{equation}\label{eq:20}
\zeta(3)=\pi^2\left(1+\frac{1}{2}\lim_{m\to \infty}\frac{1}{m}\sum_{n=1}^{m}\frac{1}{j_n}\right)^{-1}.
\end{equation}
These relations motivate to obtain the generalized zeta series for $\zeta^{-1}(w)$ by Theorem 1 using the $m$th logarithmic differentiation to obtain

\begin{equation}\label{eq:20}
\begin{aligned}
Z_{j}(m) &=\frac{1}{(m-1)!}\frac{d^{m}}{dw^{m}}\log\left[\frac{\zeta^{-1}(w)}{w+\frac{1}{2}}\right]\Bigr\rvert_{w\to 0}\\
           \\
          &= \lim_{k\to\infty}\frac{1}{k}\sum_{n=1}^{k}\frac{1}{j_n^m},
\end{aligned}
\end{equation}
where we specifically canceled the only zero with $w+\frac{1}{2}$. This leads to a generalized zeta series over just the singularities of $s_1=\zeta^{-1}(w)$, and the first few special values are:
\begin{equation}\label{eq:20}
\begin{aligned}
Z_{j}(1) &=2\left(-1+\frac{\pi^2}{\zeta(3)}\right)\\
         &= 14.42119333144247050884\ldots,\\
\\
Z_{j}(2) &=4\left(1-\frac{\pi^4}{\zeta(3)^2}-8\frac{\pi^6}{\zeta(3)^3}\zeta^{\prime\prime}(-2)\right)\\
         &= 899.16532329931876633541\ldots, \\
\\
Z_{j}(3) &=8\left(-1+\frac{\pi^6}{\zeta(3)^3}+(12\zeta^{\prime\prime}(-2)+8\zeta^{\prime\prime\prime}(-2))\frac{\pi^8}{\zeta(3)^4}+96\frac{\pi^{10}}{\zeta(3)^5}\zeta^{\prime\prime}(-2)^2\right)\\
         &= 75463.66774845673072302538\ldots, \\
\\
Z_{j}(4) &=16\Bigg[1-\frac{\pi^8}{\zeta(3)^4}-\frac{\pi^{10}}{\zeta(3)^5}\bigg(16\zeta^{\prime\prime}(-2)+\frac{32}{3}\zeta^{\prime\prime\prime}(-2)+\frac{16}{3}\zeta^{\prime\prime\prime\prime}(-2)\bigg)+\\
         & -\frac{\pi^{12}}{\zeta(3)^6}\bigg(160\zeta^{\prime\prime}(-2)^2+\frac{640}{3}\zeta^{\prime\prime}(-2)\zeta^{\prime\prime\prime}(-2)\bigg)-1280\frac{\pi^{14}}{\zeta(3)^{7}}\zeta^{\prime\prime}(-2)^3\Bigg] \\
         &= 6936470.11903064697027091228\ldots. \\
\end{aligned}
\end{equation}
These closed-form formulas were obtained using (265) in conjunction with differentiating (255) and (261). The values of $Z_{j}(m)$ are naturally normalized by a factor $\frac{1}{k}$ taken in the limit, and they converge to a finite constant, otherwise, without the $\frac{1}{k}$ factor, they would be quickly divergent. Also, as $m\to\infty$ the $Z_{j}(m)$ diverges.

As we have shown in the previous sections by equation (201), that one could obtain a formula for $j_1$ by the limit

\begin{equation}\label{eq:20}
j_1=\lim_{m\to\infty}[Z_j(m)]^{-\frac{1}{m}}
\end{equation}
and substituting (265) for $Z_{j}(m)$ we have

\begin{equation}\label{eq:20}
j_1=\lim_{m\to\infty}\left[-\frac{m}{(m-1)!}\frac{d^{m}}{dw^{m}}\log\left[\frac{\zeta^{-1}(w)}{w+\frac{1}{2}}\right]\Bigr\rvert_{w\to 0}\right]^{-\frac{1}{m}}
\end{equation}
and a factor of $m$ is needed to cancel the $\frac{1}{m}$ from $Z_{j}(m)$. Numerical computation of (268) is summarized in Table 5 in the previous section.

\section{Conclusion}
In the presented work, we utilized the $m$th log-derivative formula to obtain a generalized zeta series of the zeros of the Riemann zeta function from which we can recursively extract trivial and non-trivial zeros. We then extended the same methods as to solve an equation $\zeta(s)-w=0$ in order to obtain an inverse Riemann zeta function $s=\zeta^{-1}(w)$. We then introduced a series expansion and a branch singularity expansion formula for the inverse zeta, in which the singularities $j_n$ span a conjectured interval $(j_1,1)$. Not much is known about these quantities, yet they can describe the entire $s_1=\zeta^{-1}(w)$ (the principal branch) and many identities that follow. We further numerically explored these formulas to high precision in PARI/GP software package, and show that they do indeed converge to the inverse Riemann zeta function for various test cases, and then developed an efficient computer code to compute the inverse zeta for complex $w$-domain.

We remark that the methods presented in this article can be naturally applied to find zeros and inverses of many other functions, but each function has to be custom fitted for this method. For example, we obtain an inverse of the gamma function:

\begin{equation}\label{eq:20}
s=\Gamma^{-1}(w)=\lim_{m\to\infty} \left[-\frac{1}{(m-1)!}\frac{d^{m}}{ds^{m}}\log\left((\Gamma(s)-w)s\right)\Bigr\rvert_{s\to 0}\right]^{-\frac{1}{m}}
\end{equation}
which is the principal solution, or the inverse Bessel function of the first kind
\begin{equation}\label{eq:20}
s=J^{-1}_{\nu,n}(w)=\lim_{m\to\infty} \left[-\frac{1}{2(m-1)!}\frac{d^{m}}{ds^{m}}\log\left(J_{\nu,n}(s)-w\right)\Bigr\rvert_{s\to 0}\right]^{-\frac{1}{m}}.
\end{equation}
The Lambert-W function is defined as an inverse of $g(x)=xe^x$ which we could write as

\begin{equation}\label{eq:20}
g^{-1}(x)=W(x)=\lim_{m\to\infty}\left[-\frac{1}{(m-1)!}\frac{d^{m}}{ds^{m}}\log\left(se^s-x\right)\Bigr\rvert_{s\to 0}\right]^{-\frac{1}{m}},
\end{equation}
but one is no longer limited to $g(x)=xe^x$, and could invert essentially any function, for example, we can invert $h(x)=(x-x^3)e^{x}$ as
\begin{equation}\label{eq:20}
h^{-1}(x)=\lim_{m\to\infty}\left[-\frac{1}{(m-1)!}\frac{d^{m}}{ds^{m}}\log\left((s-s^3)e^{s}-x\right)\Bigr\rvert_{s\to 0}\right]^{-\frac{1}{m}}.
\end{equation}
One also has the inversion of trigonometric functions

\begin{equation}\label{eq:20}
\cos^{-1}(x)=\lim_{m\to\infty}\left[-\frac{1}{2(m-1)!}\frac{d^{m}}{ds^{m}}\log\left(\cos(s)-x\right)\Bigr\rvert_{s\to 0}\right]^{-\frac{1}{m}},
\end{equation}
or
\begin{equation}\label{eq:20}
\cos(x)=\lim_{m\to\infty}\left[-\frac{1}{(m-1)!}\frac{d^{m}}{ds^{m}}\log\left(\cos^{-1}(s)-x\right)\Bigr\rvert_{s\to 0}\right]^{-\frac{1}{m}}.
\end{equation}
A generalized zeta series can also be inverted, for example, the inverse of the secondary zeta function as

\begin{equation}\label{eq:20}
Z^{-1}_1(w)=\lim_{m\to\infty}\left[-\frac{1}{(m-1)!}\frac{d^{m}}{ds^{m}}\log\left(\frac{1}{t_1^{s}}+\frac{1}{t_2^{s}}+\frac{1}{t_3^{s}}+\ldots-w\right)\Bigr\rvert_{s\to 0}\right]^{-\frac{1}{m}},
\end{equation}
which is over imaginary parts of non-trivial zeros.

Finally, we give an example to how solve a finite degree polynomial of degree six. First we create a polynomial with prescribed zeros by factorization as
\begin{equation}\label{eq:20}
f(x)=(x-1)(x-2)(x-3)(x-4)(x-5)(x-6)
\end{equation}
which yields the polynomial
\begin{equation}\label{eq:20}
f(x)= x^6-21x^5+175x^4-735x^3+1624x^2-1764x+720
\end{equation}
we wish to solve. Then using Theorem 1, we obtain the generalized zeta series over its zeros as

\begin{equation}\label{eq:20}
Z(m)=-\frac{1}{(m-1)!}\frac{d^{m}}{dx^{m}}\log\bigg[x^6-21x^5+175x^4-735x^3+1624x^2-1764x+720\bigg]\Bigr\rvert_{x\to 0}.
\end{equation}
The principal zero is computed as

\begin{equation}\label{eq:20}
z_1=\lim_{m\to\infty}[Z(m)]^{-\frac{1}{m}}
\end{equation}
and a numerical computation for $m=100$ yields

\begin{equation}\label{eq:20}
z_1=0.99999999999999999999\ldots
\end{equation}
accurate to 32 decimal places, at which point we just round to $1$. The next zero is recursively found as

\begin{equation}\label{eq:20}
z_2=\lim_{m\to\infty}(Z(m)-\frac{1}{1^m})^{-\frac{1}{m}}
\end{equation}
and a numerical computed for $m=100$ yields

\begin{equation}\label{eq:20}
z_2=1.999999999999999999\underline{9}5\ldots
\end{equation}
which is accurate to 19 decimal places, at which point we just round to $2$. The next zero is recursively found as

\begin{equation}\label{eq:20}
z_3=\lim_{m\to\infty}(Z(m)-\frac{1}{1^m}-\frac{1}{2^m})^{-\frac{1}{m}}
\end{equation}
and a numerical computation for $m=100$ yields

\begin{equation}\label{eq:20}
z_3= 2.9999999999999\underline{9}037839\ldots
\end{equation}
which is accurate to 14 decimal places, at which point we just round to $3$. We keep repeating this and compute the remaining zeros
\begin{equation}\label{eq:20}
\begin{aligned}
z_4 &= 3.9999999999\underline{9}185185599\ldots, \\
z_5 &= 4.99999999\underline{9}39626633006\ldots,\\
z_6 &= 5.99999999999999999999\ldots. \\
\end{aligned}
\end{equation}
Such methods can be effectively used to solve a finite or infinite degree polynomial and is straightforward if the zeros are positive and real, but the root extraction becomes more difficult if the roots are a mix of positive and negative numbers or if they are complex. Usually, some other knowledge about these zeros would be required in order to extract them recursively. The presented methods give analytical formulas for the zeros of functions and their inverses.

\section*{Note}
This manuscript was further improved over the previously published versions. 

\section*{Acknowledgement}
I thank the referees at CMST Journal for reviewing the original manuscript, and providing valuable feedback. I also thank Dr. R.P. Brent for discussions about his recently published (BPT) theorem, as referenced in [4,5], which is key in computing the Hassani constant to high accuracy. Finally I thank the team at CMST Journal for their fantastic work, in particular, Dr. Marek Wolf and Dr. Krzysztof Wojciechowski.

\texttt{Email: art.kawalec@gmail.com}

\newpage
\section{Appendix A}
As a reference, we a generate a set of branch singularities $j_n$ for $m=50$ in Table 11.

\begin{table}[hbt!]
\caption{A generated attractor of $j_n$ branch singularities for $m=50$ (first 20 digits).}
\centering
\begin{tabular}{||c| c| c|c| c| c||}
\hline\hline
$n$  & $\Re(j_n)$ & $\Im(j_n)$ &  $n$ & $\Re(j_n)$ & $\Im(j_n)$\\ [0.5ex]
\hline
\hline
1 &	 0.00924817888645333386 & 0 &26 & 0.48044184864825218169 & 0 \\
2 &	 0.00996087442670693606 & 0 &27	& 0.52650880760811362399 & 0 \\
3 &	 0.01141938808870171357 & 0 &28	& 0.57365240977961119353 & 0 \\
4 &	 0.01368586086618980746 & 0 &29	& 0.62127161119034671556 & 0 \\
5 &	 0.01684470809898810962 & 0 &30	& 0.66867281662328202071 & 0 \\
6 &  0.02099530482694698571 & 0 &31	& 0.71509271434195433147 & 0 \\
7 &	 0.02624577567440431672 & 0 &32	& 0.75973208235517804285 & 0 \\
8 &	 0.03270868212775015136 & 0 &33	& 0.80179908715096726792 & 0 \\
9 &	 0.04049860736728589915 & 0 &34	& 0.84055880661322018634 & 0 \\
10 & 0.04973119568552762263 & 0 &35	& 0.87538416573629878458 & 0 \\
11 & 0.06052307676214534294 & 0 &36	& 0.90580259267755448615 & 0 \\
12 & 0.07299215386578867256 & 0 &37	& 0.93153277730500318545 & 0 \\
13 & 0.08725783916047022045 & 0 &38	& 0.95250698743358019409 & 0 \\
14 & 0.10344091372080303722 & 0 &39	& 0.96887620857910871910 & 0 \\
15 & 0.12166275250580498916 & 0 &40	& 0.98099739567543390881 & 0 \\
16 & 0.14204368600343120189 & 0 &41	& 0.98940487555888638730 & 0 \\
17 & 0.16470028029514330906 & 0 &42	& 0.99465572536512300752 & 0 \\
18 & 0.18974131865445660803 & 0 &43	& 1.00176360153074581721 & -0.000748412701421  \\
19 & 0.21726227453862363234 & 0 &44	& 1.00176360153074581721 &  0.000748412701421  \\
20 & 0.24733809351267049963 & 0 &45	& 0.99696259008061343773 & -0.001379208199501  \\
21 & 0.28001416776669776249 & 0 &46	& 0.99696259008061343773 &  0.001379208199501  \\
22 & 0.31529551052554821559 & 0 &47	& 1.00076852275562395685 & -0.001960850963677  \\
23 & 0.35313433733493148978 & 0 &48	& 1.00076852275562395685 &  0.001960850963677  \\
24 & 0.39341655028813465181 & 0 &49	& 0.99900506368913964681 & -0.002338536288224  \\
25 & 0.43594800026223553641 & 0 &50	& 0.99900506368913964681 &  0.002338536288224  \\
\hline
\end{tabular}
\label{table:nonlin}
\end{table}

We compute the following identities based on this data set.
\noindent The mean value is computed as

\begin{equation}\label{eq:20}
\mu=\frac{1}{m}\sum_{n=1}^{m}j_n  = \frac{1}{2}
\end{equation}
accurate to $71$ decimal places. And also the variance is

\begin{equation}\label{eq:20}
\begin{aligned}
\sigma^2&=\frac{1}{m}\sum_{n=1}^{m}(j_n-\frac{1}{2})^2 \\
&=0.15443132980306572121 \ldots
\end{aligned}
\end{equation}
and
\begin{equation}\label{eq:20}
\sigma=0.39297751819037399128\ldots
\end{equation}
The mean value of $\log(j_n)$ is

\begin{equation}\label{eq:20}
\begin{aligned}
\frac{1}{m}\sum_{n=1}^{m}\log(j_n)&= -2\log(2),\\
&=-1.386294361119890\underline{6}0107\ldots
\end{aligned}
\end{equation}
accurate to $16$ decimal places. We also have the generalized zeta series
\begin{equation}\label{eq:20}
\begin{aligned}
\frac{1}{m}\sum_{n=1}^{m}\frac{1}{j_n} &= 2\left(-1+\frac{\pi^2}{\zeta(3)}\right),\\
&=  14.421193331442\underline{4}2811203\ldots
\end{aligned}
\end{equation}
accurate to $13$ decimal places. We compute
\begin{equation}\label{eq:20}
\begin{aligned}
\prod_{n=1}^{m}\left(1+\frac{1}{2j_n}\right)^{\frac{1}{m}}&=4\log\sqrt{2\pi},\\
&=  3.675754132818690\underline{9}0182,
\end{aligned}
\end{equation}
which is accurate to $16$ decimal places, and also, when taking the limit $s=10^{20}$, the Euler-Mascheroni constant:

\begin{equation}\label{eq:20}
\begin{aligned}
\gamma &=\left[(s+\frac{1}{2})\prod_{n=1}^{m}\left(s-j_n\right)^{-\frac{1}{m}}-(1+\frac{1}{s})\right]s^2\\
       &=0.577215664901532860\underline{6}1\ldots
\end{aligned}
\end{equation}
accurate to $19$ decimal places.

\section{Appendix B}
There is a recurrence relation for the generalized zeta series $Z_{\nu}(s)$ over the zeros of the Bessel function of the first kind found in Sneddon [18, p.149], which satisfies
\begin{equation}\label{eq:20}
\sum_{r=0}^{m}\frac{m!\Gamma(m+\nu+1)}{(m-r)!\Gamma(m+\nu-r+1)}(-4)^r Z_{\nu}(2r+2)=\frac{1}{4(\nu+m+1)},
\end{equation}
from which, with additional simplifications, we re-write this formula slightly differently as to keep it compact by defining the summand term
\begin{equation}\label{eq:20}
K(r,m,\nu)=(-4)^r\frac{m!\Gamma(m+\nu+1)}{(m-r)!\Gamma(m-r+\nu+1)}
\end{equation}
so that
\begin{equation}\label{eq:20}
Z_{\nu}(2m+2) = \frac{1}{K(m,m,\nu)}\left(\frac{1}{4(m+\nu+1)}-\sum_{r=1}^{m}K(r-1,m,\nu)Z_{\nu}(2r)\right),
\end{equation}
and further expanding yields
\begin{equation}\label{eq:20}
\begin{aligned}
Z_{\nu}(2m+2) =& \frac{1}{4}\sum_{r=1}^{m}\frac{Z_{\nu}(2r)}{(-4)^{m-r}(m-r+1)!(\nu+1)_{m-r+1}}+\\
              &+(-1)^m\bigg(\frac{1}{4}\bigg)^{m+1}\frac{1}{m!(m+\nu+1)(\nu+1)_{m}},
\end{aligned}
\end{equation}
also found in [18, eq (39)]. The gamma terms of the type

\begin{equation}\label{eq:20}
(x)_{k}=\frac{\Gamma(x+k)}{\Gamma(x)}=x(x+1)(x+2)\cdots(x+k-1)=\sum_{n=0}^{k}(-1)^{n}s(n,k)x^n,
\end{equation}
is the rising factorial (which can be written in terms of the Pochhammer symbol) result in a finite $k$th degree polynomial function with integer coefficients, i.e., the Stirling numbers of the first kind $s(n,k)$. Also for the case $m=0$ the summation term in (295) assumed to be $0$ to bootstrap the recurrence. These formulas generate $Z_{\nu}(2m)$ as shown in equation (70) and Table 1.

\end{document}